\RequirePackage{ifthen}
\RequirePackage{ifpdf}
\newcommand{\driverOption}{}
\ifthenelse{\boolean{pdf}}{
  \renewcommand{\driverOption}{pdftex}
} { % else
  \renewcommand{\driverOption}{dvips}
}

\documentclass[reqno, 11pt, letterpaper, commented, oneside, \driverOption]{amsart}

\newboolean{isCommented}
\usepackage{geometry}
\usepackage{bbm}
\usepackage[final]{graphicx}
\usepackage{color}
\usepackage{upref}
\usepackage{enumerate}
\usepackage{latexsym}
\usepackage{amssymb}
\usepackage[ansinew]{inputenc}
\usepackage[T1]{fontenc}
\usepackage[ruled,vlined,norelsize]{algorithm2e}  % norelsize needed to compile on arxiv
\usepackage{mathtools}
\usepackage{mycommands}

\ifthenelse{\boolean{pdf}}{
  \usepackage[final]{ps4pdf}
} { % else
  \usepackage[inactive]{ps4pdf}
}
\PSforPDF{\usepackage{psfrag}}

\newcommand{\hyperrefDriverOption}{}
\ifthenelse{\boolean{pdf}}{
	\renewcommand{\hyperrefDriverOption}{pdftex}
} { % else
	\renewcommand{\hyperrefDriverOption}{hypertex}
}
\usepackage[\hyperrefDriverOption,
  colorlinks = false, 
  pdfauthor={Torsten M\"utze, Christoph Standke, Veit Wiechert},
  pdftitle={A minimum-change version of the Chung-Feller theorem for Dyck paths}]
  {hyperref}
    
\ifthenelse{\boolean{isCommented}} {
	\newcommand{\TM}[1]{\marginpar{\parbox{4cm}{{\small {\bf TM:} #1}}}}
	\newcommand{\VW}[1]{\marginpar{\parbox{4cm}{{\small {\bf FW:} #1}}}}
} { % else
	\newcommand{\TM}[1]{}
	\newcommand{\VW}[1]{}
}
\newtheorem{theorem}{Theorem}
\newtheorem{lemma}[theorem]{Lemma}

\theoremstyle{definition}

\theoremstyle{remark}

\long\def\symbolfootnote[#1]#2{\begingroup
\def\thefootnote{\fnsymbol{footnote}}\footnote[#1]{#2}\endgroup}

\ifthenelse{\boolean{isCommented}} {
	\geometry{%showframe,
	  hmargin={25mm, 50mm},
	  marginparwidth=40mm, 
	  vmargin={25mm, 25mm},
	  headsep=10mm,
	  headheight=5mm,
	  footskip=10mm
	}
} { % else
	\geometry{%showframe,
	  hmargin={25mm, 25mm}, 
	  vmargin={25mm, 25mm},
	  headsep=10mm,
	  headheight=5mm,
	  footskip=10mm
	}
}

\begin{document}

\begin{center}

\LARGE A minimum-change version of the \\ Chung-Feller theorem for Dyck paths
\vspace{2mm}

\Large Torsten M\"utze, Christoph Standke, Veit Wiechert
\vspace{2mm}

\large
  Institut f\"ur Mathematik \\
  TU Berlin, 10623 Berlin, Germany \\
  {\small\tt \{muetze,standke,wiechert\}@math.tu-berlin.de}
\vspace{5mm}

\small

\begin{minipage}{0.8\linewidth}
\textsc{Abstract.}
A Dyck path with $2k$ steps and $e$ flaws is a path in the integer lattice that starts at the origin and consists of $k$ many $\upstep$-steps and $k$ many $\downstep$-steps that change the current coordinate by $(1,1)$ or $(1,-1)$, respectively, and that has exactly $e$ many $\downstep$-steps below the line $y=0$. Denoting by $D_{2k}^e$ the set of Dyck paths with $2k$ steps and $e$ flaws, the Chung-Feller theorem asserts that the sets $D_{2k}^0,D_{2k}^1,\ldots,D_{2k}^k$ all have the same cardinality $\frac{1}{k+1}\binom{2k}{k}=C_k$, the $k$-th Catalan number.
The standard combinatorial proof of this classical result establishes a bijection $f'$ between $D_{2k}^e$ and $D_{2k}^{e+1}$ that swaps certain parts of the given Dyck path $x$, with the effect that $x$ and $f'(x)$ may differ in many positions. In this paper we strengthen the Chung-Feller theorem by presenting a simple bijection $f$ between $D_{2k}^e$ and $D_{2k}^{e+1}$ which has the additional feature that $x$ and $f(x)$ differ in only \emph{two} positions (the least possible number). We also present an algorithm that allows to compute a sequence of applications of $f$ in constant time per generated Dyck path.
As an application, we use our minimum-change bijection $f$ to construct cycle-factors in the odd graph $O_{2k+1}$ and the middle levels graph $M_{2k+1}$ --- two intensively studied families of vertex-transitive graphs --- that consist of $C_k$ many cycles of the same length.
\end{minipage}

\end{center}

\vspace{5mm}

%\tableofcontents

\section{Introduction}

The Catalan numbers are one of the most fundamental counting sequences in combinatorics, and Dyck paths are among the most heavily studied Catalan families (see \cite{stanley-cat:15} and references therein).
A \emph{Dyck path with $2k$ steps and $e$ flaws} ($e\in\{0,1,\ldots,k\}$) is a path in the integer lattice $\mathbb{Z}^2$ that starts at the origin $(0,0)$ and consists of $k$ many $\upstep$-steps and $k$ many $\downstep$-steps that change the current coordinate by $(1,1)$ or $(1,-1)$, respectively (so the path ends at the coordinate $(2k,0)$), and that has exactly $e$ many $\downstep$-steps below the line $y=0$. In particular, a Dyck path with zero flaws never moves below the line $y=0$, and a Dyck path with $k$ flaws never moves above the line $y=0$.
We denote the set of all Dyck paths with $2k$ steps and $e$ flaws by $D_{2k}^e$. Clearly, the total number of Dyck paths with $2k$ steps and an arbitrary number of flaws is $\binom{2k}{k}$. The classical Chung-Feller theorem, first proved in \cite[p.~168]{MacMahon:09}, asserts that partitioning the set of all these Dyck paths according to the number of flaws yields $k+1$ subsets of the same size $\frac{1}{k+1}\binom{2k}{k}=:C_k$, where $C_k$ is the $k$-th Catalan number.

\begin{theorem}[\cite{MacMahon:09}, \cite{chung-feller:49}]
\label{thm:chung-feller}
For any $k\geq 1$ we have $|D_{2k}^0|=|D_{2k}^1|=\cdots=|D_{2k}^k|=\frac{1}{k+1}\binom{2k}{k}=C_k$.
\end{theorem}

% Chung-Feller 1949 \cite{chung-feller:49}: first appearance and proof of the theorem in terms of probabilities in coin-tossing, proof using analytic methods
% Hodges 1955 \cite{hodges:55}: simple bijective proof, essentially the same as Chen's later proof
% Narayana 1967 \cite{MR0224486}: combinatorial proof using cyclic permutations
% Chen 2008 \cite{MR2382368}: simple combinatorial proof via cut-and-paste technique on Dyck paths, bijection that increases the number of flaws by 1
% Ruckavicka 2011 \cite{MR2776816}: generalization of Chen's proof to Dyck paths that have horizontal steps of length possibly greater than 1, bijection that increases the number of flaws by 1
% Eu, Fu, Yeh 2005 \cite{MR2167479}: refined version of Chung-Feller theorem (subpartitioning of sets of Dyck paths with a fixed number of flaws according to number of segments below line $y=x$), bijective proof for this theorem
% Ma, Yeh 2009 \cite{MR2571698}: generalization of Chung-Feller theorems using analytic method
% Woan 2001 \cite{MR1840664}: Chung-Feller like partition theorems for lattice paths based on the number of upsteps after the rightmost lowest point

Since its discovery, several proofs and generalizations of Theorem~\ref{thm:chung-feller} have appeared in the literature \cite{MacMahon:09, chung-feller:49, hodges:55, MR0224486, MR2382368, MR2776816, MR1840664, MR2167479, MR2571698}.
The standard combinatorial proof, first presented in \cite{hodges:55} and rediscovered in \cite{MR2382368}, establishes the following bijection $f'$ between the sets $D_{2k}^e$ and $D_{2k}^{e+1}$ (see Figure~\ref{fig:gx}): For any Dyck path $x\in D_{2k}^e$, considering the first $\upstep$-step above the line $y=0$ and the first $\downstep$-step returning to the line $y=0$, and denoting the corresponding indices by $a$ and $b$, the path $x=(x_1,x_2,\ldots,x_{2k})$ can be partitioned as $x=u\circ \upstep \circ v\circ \downstep\circ w$ where $u:=(x_1,\ldots,x_{a-1})$, $v:=(x_{a+1},\ldots,x_{b-1})$ and $w:=(x_{b+1},\ldots,x_{2k})$, and $\circ$ denotes the concatenation. Then the image of $x$ under $f'$ is defined as $f'(x):=v\circ \downstep\circ u\circ\upstep\circ w$.
It is easy to check that $f'$ increases the number of flaws by 1 and that it is indeed a bijection between $D_{2k}^e$ and $D_{2k}^{e+1}$ (see Figure~\ref{fig:chung-feller-bij}). Note however that the number of positions in which $x=u\circ \upstep \circ v\circ \downstep\circ w$ and $f'(x)=v\circ \downstep\circ u\circ\upstep\circ w$ differ can be large due to the swapping of the subpaths $u$ and $v$. In the worst case the two Dyck paths may differ in every position, see e.g.\ the paths $x'$ and $f'(x')$ on the right-hand side of Figure~\ref{fig:chung-feller-bij}.

\subsection{Our results}
\label{sec:results}

\begin{figure}
\centering
\PSforPDF{
 \psfrag{u}{$u$}
 \psfrag{v}{$v$}
 \psfrag{w}{$w$}
 \psfrag{x}{$x$}
 \psfrag{fsx}{$f'(x)$}
 \includegraphics{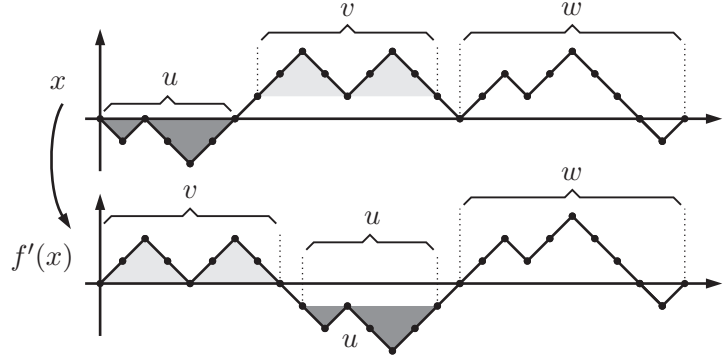}
}
\caption{Bijection $f'$ used to prove Theorem~\ref{thm:chung-feller}. The Dyck path $x$ has $4$ flaws, the Dyck path $f'(x)$ has $5$ flaws.}
\label{fig:gx}
\end{figure}

\begin{figure}
\centering
\PSforPDF{
 \psfrag{x}{$x$}
 \psfrag{xp}{$x'=g(x)$}
 \psfrag{xpp}{$f(x)=h(x')$}
 \psfrag{g}{$g$}
 \psfrag{h}{$h$}
 \psfrag{d0x}{$d_0(x)=3$}
 \psfrag{u1xp}{$u_1(x')=6$}
 \includegraphics{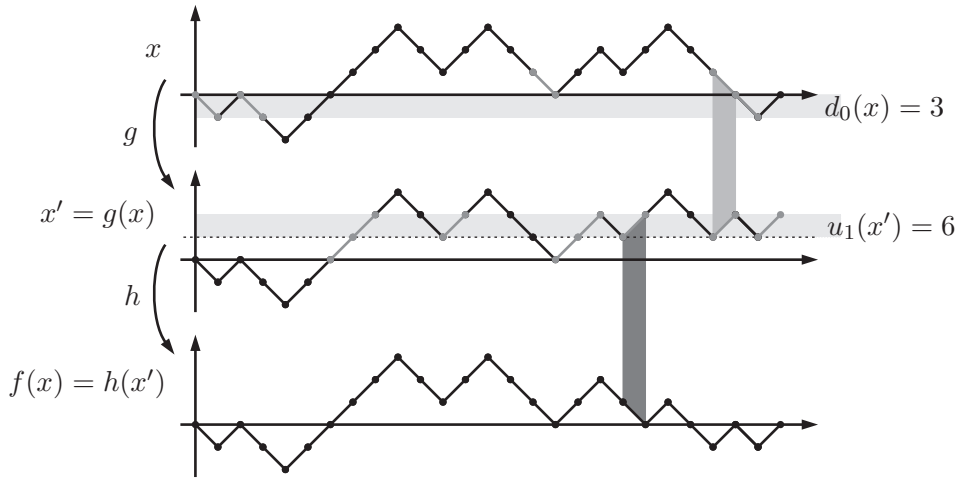}
}
\caption{Definition of the bijection $f$ from Theorem~\ref{thm:min-change-bij}. The Dyck path $x$ has $4$ flaws, the Dyck path $f(x)=h(g(x))$ has $5$ flaws, and $x$ and $f(x)$ differ in only two positions marked with gray vertical bars. In the upper part of the figure, all $\downstep$-steps of $x$ touching the line $y=0$ are drawn gray. In the middle part, all $\upstep$-steps of $x'$ touching the line $y=1$ are drawn gray.}
\label{fig:fx}
\end{figure}

\begin{figure}
\centering
\PSforPDF{
 \psfrag{xp}{$x'$}
 \psfrag{fsxp}{$f'(x')$}
 \psfrag{fs}{$f'$}
 \psfrag{d60}{\Large $D_6^0$}
 \psfrag{d61}{\Large $D_6^1$}
 \psfrag{d62}{\Large $D_6^2$}
 \psfrag{d63}{\Large $D_6^3$}
 \includegraphics{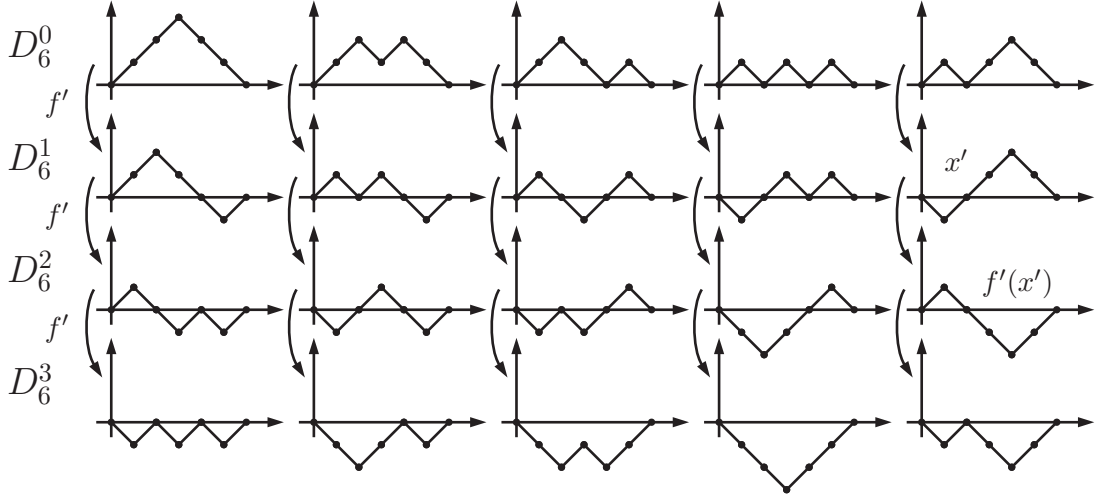}
}
\caption{Illustration of the bijection $f'$ used to prove Theorem~\ref{thm:chung-feller} on the set of all Dyck paths of length 6.}
\label{fig:chung-feller-bij}
\end{figure}

\begin{figure}
\centering
\PSforPDF{
 \psfrag{f}{$f$}
 \psfrag{d60}{\Large $D_6^0$}
 \psfrag{d61}{\Large $D_6^1$}
 \psfrag{d62}{\Large $D_6^2$}
 \psfrag{d63}{\Large $D_6^3$}
 \psfrag{a}{i}
 \psfrag{b}{ii}
 \psfrag{c}{iii}
 \psfrag{d}{iv}
 \psfrag{e}{v}
 \psfrag{01}{1}
 \psfrag{02}{2}
 \psfrag{03}{3}
 \psfrag{04}{4}
 \psfrag{05}{5}
 \psfrag{06}{6}
 \psfrag{07}{7}
 \psfrag{08}{8}
 \psfrag{09}{9}
 \psfrag{10}{10}
 \psfrag{a1}{11}
 \psfrag{12}{12}
 \psfrag{13}{13}
 \psfrag{14}{14}
 \psfrag{15}{15}
 \psfrag{16}{16}
 \psfrag{17}{17}
 \psfrag{18}{18}
 \psfrag{19}{19}
 \psfrag{20}{20}
 \psfrag{p1}{\small (6,2,4,3,5,1)}
 \psfrag{p2}{\small (6,4,5,2,3,1)}
 \psfrag{p3}{\small (4,2,3,1,6,5)}
 \psfrag{p4}{\small (2,1,4,3,6,5)}
 \psfrag{p5}{\small (2,1,6,4,5,3)}
 \includegraphics{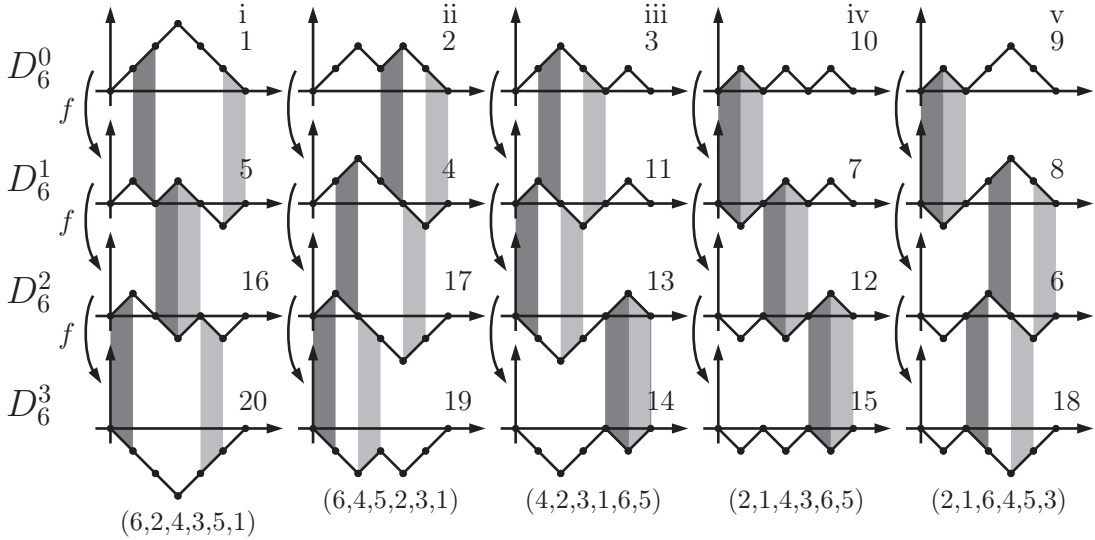}
}
\caption{Illustration of the minimum-change bijection $f$ from Theorem~\ref{thm:min-change-bij} on the set of all Dyck paths of length 6. The gray bars highlight the two positions that are flipped in each step, and the sequences at the bottom the entire sequence of positions that are flipped when repeatedly applying $f$. The numbers 1--20 and i--v, respectively, indicate related minimum-change orderings of all Dyck paths (with an arbitrary number of flaws) and those with zero flaws as explained in Section~\ref{sec:algo}.}
\label{fig:min-change-bij}
\end{figure}

The main result of this paper is the following strengthening of Theorem~\ref{thm:chung-feller}: We establish a bijection $f$ between $D_{2k}^e$ and $D_{2k}^{e+1}$ such that for any $x\in D_{2k}^e$, the Dyck paths $x$ and $f(x)$ differ in only \emph{two} positions (two changes are clearly the least possible number). In other words: $f(x)$ is obtained from $x$ by flipping a single $\upstep$-step and a single $\downstep$-step, leaving the initial and terminal parts of the path unchanged, and shifting down the middle part by two units.

This minimum-change bijection $f$ is defined as follows (see Figure~\ref{fig:fx}): For any lattice path $x$ and any $c\in\mathbb{Z}$, we let $u_c(x)$ and $d_c(x)$ denote the number of $\upstep$-steps or $\downstep$-steps, respectively, of $x$ starting on the line $y=c$. We say that an $\upstep$-step or $\downstep$-step of $x$ \emph{touches} the line $y=c$, if it starts or ends on this line.
For any $x\in D_{2k}^e$, $e<k$, we let $x':=g(x)$ denote the lattice path obtained from $x$ by flipping the $(d_0(x)+1)$-th $\downstep$-step of $x$ touching the line $y=0$. Similarly, we let $h(x')$ denote the lattice path obtained from $x'$ by flipping the $u_1(x')$-th $\upstep$-step of $x'$ touching the line $y=1$. We then define $f(x):=h(x')=h(g(x))$.
Figure~\ref{fig:min-change-bij} shows how $f$ acts on the set of all Dyck paths of length 6 (the two flipped positions are highlighted with gray vertical bars).

\begin{theorem}
\label{thm:min-change-bij}
For any $k\geq 1$ and any $e\in\{0,1,\ldots,k-1\}$ the mapping $f$ is a bijection between $D_{2k}^e$ and $D_{2k}^{e+1}$ such that for any $x\in D_{2k}^e$, the Dyck paths $x$ and $f(x)$ differ in exactly two positions.
\end{theorem}

\subsubsection{Efficient algorithms}
\label{sec:algo}

We also present an algorithm that allows to compute repeated applications of the minimum-change bijection $f$ from Theorem~\ref{thm:min-change-bij} efficiently.

\begin{theorem}
\label{thm:algo}
There is an algorithm that computes for any given $k\geq 1$, $e\in\{0,1,\ldots,k-1\}$ and $x\in D_{2k}^e$ each of the Dyck paths $f(x),f^2(x),\ldots,f^{k-e}(x)$ in time $\cO(1)$ and space $\cO(k)$ (overall time including initialization is $\cO(k)$).
\end{theorem}

Being able to generate each Dyck path in constant time is clearly best possible. We implemented this algorithm in C++, and we invite the reader to experiment with this code, which can be downloaded from the authors' websites \cite{www}.

Under this algorithmic lense, the minimum-change bijection $f$ can be seen as a \emph{combinatorial Gray code}. This term refers to the problem of efficiently generating all objects in a particular combinatorial class --- e.g. permutations, partitions, subsets, strings, trees, Dyck paths etc. --- which is a fundamental task in the area of combinatorial algorithms, as such generation algorithms are used as building blocks in a wide range of practical applications (the survey \cite{MR1491049} lists numerous references, see also \cite{MR0396274,MR993775}).
Let us add a few remarks and pointers to related work on combinatorial Gray codes that put the results of this paper into perspective.
Gray codes are named after Frank Gray, who invented a method to generate all $2^n$ subsets of $\{1,2,\ldots,n\}$ such that any two consecutive subsets differ by adding or removing a single element \cite{gray:patent} (see also \cite[Section 7.2.1.1]{knuth}). Clearly, this problem is equivalent to generating all $2^n$ bitstrings of length $n$ such that any two consecutive bitstrings differ in exactly one bit.
A closely related problem that has received considerable attention in the literature is to generate all subsets of $\{1,2,\ldots,n\}$ of size exactly $k$ for some $k\leq n$, such that any two consecutive subsets differ by exchanging a single element \cite{MR0349274,MR0366085,MR0424386} (see also \cite[Section 7.2.1.3]{knuth}). Observe that this is equivalent to generating all bitstrings of length $n$ with exactly $k$ entries equal to 1 such that any two consecutive bitstrings differ in exactly two bits (a 0-bit and a 1-bit are swapped).
A special case are the Dyck paths of length $2k$ (with an arbitrary number of flaws) considered in this paper, which can be interpreted as a subset of $\{1,2,\ldots,2k\}$ of size $k$ or a bitstring of length $2k$ with $k$ entries equal to 1 (every $\upstep$-step corresponds to a 1-bit, and every $\downstep$-step to a 0-bit).
Combinatorial Gray codes for these $(n,k)$-combinations can be computed efficiently even if we impose further constraints on the allowed swaps, e.g.\ we may only allow swaps of a 0-bit and 1-bit when all bits in between are equal to 0 \cite{MR782221}. Even more restrictively, we may only allow swaps of bits that are at most 2 positions apart \cite{MR995888,MR1352777}, or only between adjacent bits \cite{MR821383,MR936104}. In the latter case $n$ and $k$ are required to be even and odd, respectively, or $k
\in\{0,1,n-1,n\}$. In particular, for odd $k$ we may generate all Dyck paths of length $2k$ (with an arbitrary number of flaws) using only such adjacent swaps (a valley $\upstep\downstep$ becomes a peak $\downstep\upstep$ or vice versa; such an ordering 1--20 is indicated in Figure~\ref{fig:min-change-bij} for $k=3$).

% Knuth \cite[Section 7.2.1.3]{knuth} section of generating (n,k)-combinations (he calls them (k,n-k)-combinations)
% Tang/Liu 1973 \cite{MR0349274}: algorithm to generate combinations by deleting all irrelevant entries from general Gray code
% Ehrlich 1973 \cite{MR0366085}, Bitner/Ehrlich/Reingold 1976 \cite{MR0424386}: loopless algorithms for permutations, combinations (all \binom{n}{k} sets)
% Eades/McKay 1984 \cite{MR782221}: algorithm to generate all k-subsets of [n] only piano exchanges allowed (this is weaker than the  [Eades/Hickey/Read84] and the [Ruskey88] result)
% Jenkyns/McCarthy 1995 \cite{MR1352777}: algorithm where only transitions 01<-->10 and 001<-->100 are allowed
% Chase 1990 \cite{MR995888}: loopless and efficient algorithm for the [Jenkyns/McCarthy95] result
% Eades/Hickey/Read 1984 \cite{MR821383}: existence result for the [Ruskey88] result
% Ruskey 1988 \cite{MR936104}: adjacent interchange algorithm: only transitions 01<-->10 are allowed, generate all k-subsets of [n], algorithm exists iff [k=0,1,n-1,n] or [n is even and k is odd] (the special case n=2k with odd k is included)
% Ruskey/Williams 2009 \cite{MR2548545}: generate all combinations by simple prefix rotation algorithm

\begin{figure}
\centering
\PSforPDF{
 \psfrag{d0}{\large $D_6^0$}
 \psfrag{i1}{(a)}
 \psfrag{i2}{(b)}
 \psfrag{i3}{(c)}
 \psfrag{i4}{(d)}
 \psfrag{i5}{(e)}
 \psfrag{p01}{\texttt{((()))}}
 \psfrag{p02}{\texttt{(()())}}
 \psfrag{p03}{\texttt{(())()}}
 \psfrag{p04}{\texttt{()()()}}
 \psfrag{p05}{\texttt{()(())}}
 \includegraphics{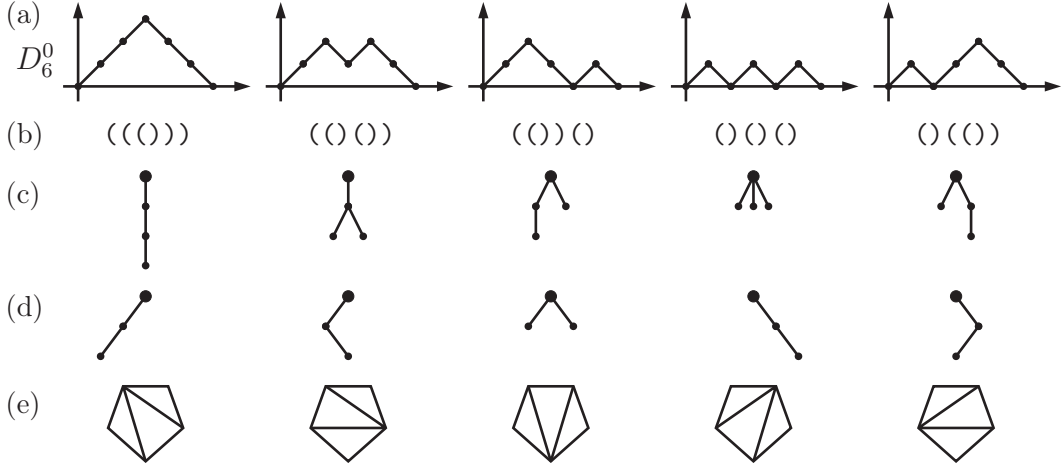}
}
\caption{Representation of (a) Dyck paths with $2k$ steps and zero flaws for $k=3$ as (b) well-formed parentheses expressions consisting of $k$ pairs of parentheses, (c) ordered trees with $k$ edges, (d) binary trees with $k$ vertices, (e) triangulations of a convex $(k+2)$-gon.}
\label{fig:catalan}
\end{figure}

We now restrict our attention to Dyck paths with $2k$ steps and zero flaws, which via standard Catalan bijections can also be viewed as well-formed parentheses expressions consisting of $k$ pairs of parentheses, (rooted) ordered trees with $k$ edges, (rooted) binary trees with $k$ vertices, triangulations of a convex $(k+2)$-gon etc.\ (see Figure~\ref{fig:catalan}, all the bijections between these classes and many more are treated in depth in \cite{stanley-cat:15}). Constant average time algorithms for generating all Dyck paths in $D_{2k}^0$ such that any two consecutive paths differ in exactly two positions have been developed in \cite{MR1041167,MR1645246}. As before, we may even impose the restriction of only allowing swaps of a 0-bit and 1-bit when all bits in between are equal to 0, or only swaps between adjacent bits \cite{MR1664740}, where in the latter case $k$ is required to be even or strictly less than 5 (such an ordering with roman numbers i--v is indicated in Figure~\ref{fig:min-change-bij} for 
$k=3$).
Summarizing, the minimum-change bijection $f$ from Theorems~\ref{thm:min-change-bij} and \ref{thm:algo} nicely complements the minimum-change generation algorithms for Dyck paths (or parentheses expressions, ordered trees, binary trees, triangulations etc.) known from the literature. In a sense, our bijection is orthogonal to the combinatorial Gray codes mentioned before and indicated in Figure~\ref{fig:min-change-bij} (the bijection $f$ operates on the columns in this figure).
We will show that for each $x\in D_{2k}^0$, the Dyck path $f^k(x)$ is obtained by mirroring $x$ at the horizontal line $y=0$ (see Figure~\ref{fig:min-change-bij}, where the Dyck paths at the top and bottom of each column are mirror images of each other).
It follows that our bijection can be combined with the before-mentioned algorithms that generate the Dyck paths in $D_{2k}^0$ and those in $D_{2k}^k$ to obtain new combinatorial Gray codes that enumerate all Dyck paths (with an arbitrary number of flaws) by flipping two positions at a time.
Such a Gray code is obtained by iterating through columns and rows in Figure~\ref{fig:min-change-bij} in a saw-tooth fashion. Specifically, $f$ and $f^{-1}$ are applied to move along the columns, and the before-mentioned generation algorithms for $D_{2k}^0$ and $D_{2k}^k$ are used to move along the first and last row.
Similarly, we obtain new combinatorial Gray codes that enumerate all the Dyck paths in $D_{2k}^e$ for some $e\in\{0,1,\ldots,k\}$ by flipping $4\cdot \min\{e,k-e\}+2$ positions at a time (again, we move through Figure~\ref{fig:min-change-bij} in a saw-tooth fashion).

% Proskurowski/Ruskey 1990 \cite{MR1041167}: constant average time generation of all Dyck words (viewed as trees) such that consecutive trees differ in one bit transposition (as in our problem)
% Walsh 1998 \cite{MR1645246}: nonrecursive and loopless implementation of Ruskey/Proskurowski's algorithm
% Bultena/Ruskey 1998 \cite{MR1664740}: generation algorithm for all Dyck words when only piano changes are allowed
% Vajnovszki/Walsh 2006 \cite{MR2577686}: listing of generalized Dyck words where only bitflips are allowed at a time

\subsubsection{Applications}

We conclude this section by presenting two concrete applications of Theorem~\ref{thm:min-change-bij}.
To discuss them we consider the \emph{odd graph} $O_{2k+1}$, defined for any integer $k\geq 1$ as the graph that has as vertices all $k$-element subsets of $\{1,2,\ldots,2k+1\}$, with an edge between any two sets that are disjoint.
Similarly, the \emph{middle levels graph} $M_{2k+1}$ has as vertices all $k$-element and all $(k+1)$-element subsets of $\{1,2,\ldots,2k+1\}$, with an edge between any two sets where one is a subset of the other (see Figure~\ref{fig:kneser}).

The odd graph and the middle levels graph are families of vertex-transitive graphs that have long been conjectured to have a Hamilton cycle for every value of $k\geq 1$ with one notable exception, the Petersen graph $O_5$ (these conjectures originated in \cite{MR0457282,MR556008} and \cite{MR737021,MR737262}, respectively, and they are mentioned e.g.\ in the popular books \cite{MR2034896,knuth,MR2858033}).
Despite the simple definition and despite considerable efforts of many researchers during that last 40 years (see \cite{MR0389663,MR510592,MR888679,MR1778200,MR1883565,MR1999733,MR2046083,MR2836824} and \cite{savage:93,MR1350586,MR1329390,MR2046083,horakEtAl:05,Gregor20102448,MR2548541,shimada-amano}, respectively), these conjectures remained open for very long (a detailed account of the historic developments can be found in \cite{muetze-su:15}). Only recently, it has been shown that the middle levels graph $M_{2k+1}$ indeed has a Hamilton cycle for every $k\geq 1$ \cite{MR3483129} (see also \cite{MR3446435,muetze-nummenpalo:16}).
Complementing earlier results \cite{MR962224,MR962223,MR1268348,MR2128031,muetze-weber:12,MR3107997} in this direction, we ask more generally, for which divisors $\ell$ of $n:=|V(G)|$ does one of these graphs $G$ have a \emph{$\cC_\ell$-factor}, i.e., a collection of $n/\ell$ many vertex-disjoint cycles, each of length $\ell$ ($V(G)$ denotes the vertex set of $G$)? It turns out that the only general case where a positive answer to this question is known is the before-mentioned Hamiltonicity of the middle levels graph (the case $\ell=n$).
We use the minimum-change bijection $f$ from Theorem~\ref{thm:min-change-bij} to produce two more positive answers (for values $\ell<n$).

\begin{figure}
\centering
\PSforPDF{
 \psfrag{o5}{\Large $O_5$}
 \psfrag{m5}{\Large $M_5$}
 \psfrag{s12}{\small $\{1,2\}$}
 \psfrag{s13}{\small $\{1,3\}$}
 \psfrag{s14}{\small $\{1,4\}$}
 \psfrag{s34}{\small $\{3,4\}$}
 \psfrag{s24}{\small $\{2,4\}$}
 \psfrag{s23}{\small $\{2,3\}$}
 \psfrag{s25}{\small $\{2,5\}$}
 \psfrag{s35}{\small $\{3,5\}$}
 \psfrag{s45}{\small $\{4,5\}$}
 \psfrag{s15}{\small $\{1,5\}$}
 \psfrag{c12}{\small $\{3,4,5\}$}
 \psfrag{c13}{\small $\{2,4,5\}$}
 \psfrag{c14}{\small $\{2,3,5\}$}
 \psfrag{c34}{\small $\{1,2,5\}$}
 \psfrag{c24}{\small $\{1,3,5\}$}
 \psfrag{c23}{\small $\{1,4,5\}$}
 \psfrag{c25}{\small $\{1,3,4\}$}
 \psfrag{c35}{\small $\{1,2,4\}$}
 \psfrag{c45}{\small $\{1,2,3\}$}
 \psfrag{c15}{\small $\{2,3,4\}$}
 \includegraphics{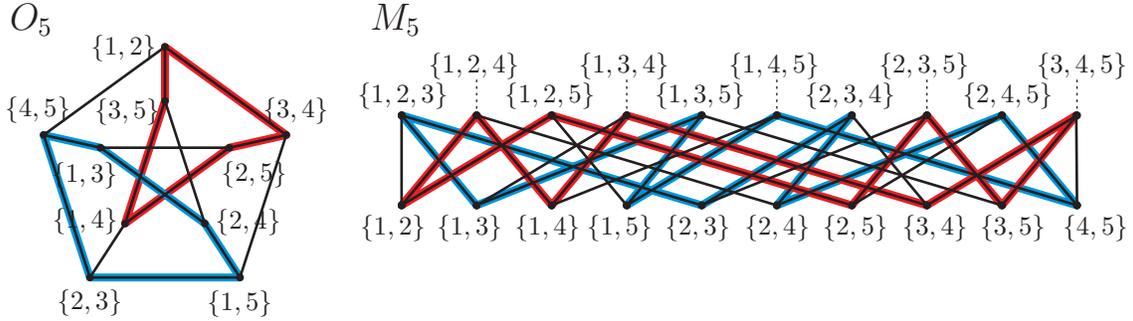}
}
\caption{Illustration of the odd graph $O_{2k+1}$ (left) and the middle levels graph $M_{2k+1}$ (right) for $k=2$, and a $\cC_{2k+1}$-factor and a $\cC_{4k+2}$-factor in these graphs (each consisting of two cycles, one colored red and one colored blue), respectively,   constructed from our minimum-change bijection $f$.}
\label{fig:kneser}
\end{figure}

\begin{theorem}
\label{thm:odd-2f}
For any $k\geq 1$, the odd graph $O_{2k+1}$ has a $\cC_{2k+1}$-factor.
\end{theorem}

\begin{theorem}
\label{thm:middle-2f}
For any $k\geq 1$, the middle levels graph $M_{2k+1}$ has a $\cC_{4k+2}$-factor.
\end{theorem}

Note that the $\cC_{2k+1}$-factor of the odd graph and the $\cC_{4k+2}$-factor of the middle levels graph each consist of $C_k$ (the $k$-th Catalan number) many cycles (see Figure~\ref{fig:kneser}).
It follows that these cycle-factors have exponentially (in $k$) many cycles of linear length.
It remains an open problem whether these cycles can be joined in a straightforward way to e.g.\ a Hamilton cycle.

\section{Proofs}

In this section we present the proofs of Theorems~\ref{thm:min-change-bij}--\ref{thm:middle-2f}.
We begin by establishing several lemmas that capture various key properties of the mappings $g$ and $h$ defined in Section~\ref{sec:results} (recall that the minimum-change bijection $f$ is the composition of $g$ and $h$).

\subsection{Combinatorial properties of the mappings $g,h$ and $f$}

In the following we consider the set $L_{2k,k}$ of all lattice paths lattice with $2k$ steps and exactly $k$ many $\upstep$-steps. These paths end at the point $(2k,0)$, and we have $L_{2k,k}=D_{2k}^0\cup D_{2k}^1\cup\cdots \cup D_{2k}^k$.
We also consider the set $L_{2k,k+1}$ of all lattice paths with $2k$ steps and exactly $k+1$ many $\upstep$-steps. These paths clearly end at the point $(2k,2)$.
Our proofs deal with four bijections between the following sets of lattice paths
\begin{align*}
  g  \colon& L_{2k,k}\setminus D_{2k}^k\rightarrow L_{2k,k+1} \enspace, \\
  g' \colon& L_{2k,k}\setminus D_{2k}^0\rightarrow L_{2k,k+1} \enspace, \\
  h  \colon& L_{2k,k+1}\rightarrow L_{2k,k}\setminus D_{2k}^0 \enspace, \\
  h' \colon& L_{2k,k+1}\rightarrow L_{2k,k}\setminus D_{2k}^k \enspace.
\end{align*}

Each of these bijections flips exactly one step of the lattice path that is provided as a function argument: $g$ and $g'$ replace a $\downstep$-step by an $\upstep$-step, and $h$ and $h'$ replace an $\upstep$-step by a $\downstep$-step. Recall the definitions of the mappings $g$ and $h$ from Section~\ref{sec:results}. The mappings $g'$ and $h'$ are defined very similarly as follows: For any $x\in L_{2k,k}\setminus D_{2k}^0$, we let $g'(x)$ denote the lattice path obtained from $x$ by flipping the $d_0(x)$-th $\downstep$-step of $x$ touching the line $y=0$ (in the definition of $g(x)$, we instead flip the $(d_0(x)+1)$-th $\downstep$-step). Similarly, for any $x\in L_{2k,k+1}$, we let $h'(x)$ denote the lattice paths obtained from $x$ by flipping the $(u_1(x)+1)$-th $\upstep$-step of $x$ touching the line $y=1$ (in the definition of $h(x)$, we instead flip the $u_1(x)$-th $\upstep$-step).
For any of the four mappings $\gamma\in\{g,g',h,h'\}$ from before and any lattice path $x$ from the corresponding domain, we let $\pos(\gamma,x)\in\{1,2,\ldots,2k\}$ denote the position in which the lattice paths $x$ and $\gamma(x)$ differ.

Note that $g$ is well-defined, as a lattice path $x\in L_{2k,k}\setminus D_{2k}^k$ has less than $k$ flaws and therefore at least one $\downstep$-step that ends on the line $y=0$ (which does not contribute to $d_0(x)$), so $x$ has at least $(d_0(x)+1)$ many $\downstep$-steps touching the line $y=0$. Similarly, $h'$ is well-defined, as a lattice path $x\in L_{2k,k+1}$ starts at $(0,0)$ and ends at $(2k,2)$ and therefore has at least one $\upstep$-step that ends on the line $y=1$ (which does not contribute to $u_1(x)$), so $x$ has at least $(u_1(x)+1)$ many $\upstep$-steps touching the line $y=1$. One can argue along very similar lines to show that $g'$ and $h$ are also well-defined.

Recall from Section~\ref{sec:results} that the function $f$ is defined for any lattice path $x\in L_{2k,k}\setminus D_{2k}^k$ as $f(x):=h(g(x))$.

Our first two lemmas assert that all the four mappings defined before are bijections, and identify the two pairs of mappings that are inverse to each other.

\begin{lemma}
\label{lem:g-inverse}
The mappings $g$ and $h'$ are bijections between the sets $L_{2k,k}\setminus D_{2k}^k$ and $L_{2k,k+1}$ and we have $g^{-1}=h'$.
\end{lemma}

\begin{proof}
Consider the lattice path $x\in L_{2k,k}\setminus D^k_{2k}$, its image $x':=g(x)$ and the position $i:=\pos(g,x)$ where they differ.
We let $x_\ell$ and $x_r$, respectively, denote the subpaths of $x$ strictly to the left and to the right of position $i$. Similarly, we let $x_\ell'$ and $x_r'$, respectively, denote the subpaths of $x'$ strictly to the left and to the right of position $i$.
We clearly have
\begin{equation}
\label{eq:xell}
  x_\ell'=x_\ell \enspace,
\end{equation}
and as the lattice path $x_r'$ is obtained by shifting $x_r$ up by two units we have
\begin{equation}
\label{eq:u1xr}
  u_1(x_r')=u_{-1}(x_r) \enspace.
\end{equation}
By the definition of $g$ we have
\begin{equation}
\label{eq:d0xlr}
  d_0(x)=d_0(x_\ell)+d_1(x_\ell) \enspace.
\end{equation}
We distinguish the cases whether the $\downstep$-step at position $i$ of $x$ starts on the line $y=0$ or on the line $y=1$.
We first consider the case that this $\downstep$-step starts on the line $y=0$.
In this case we have
\begin{align}
  d_0(x)  &= d_0(x_\ell)+d_0(x_r)+1 \enspace, \label{eq:d0x} \\
  u_1(x') &= u_1(x_\ell')+u_1(x_r') \enspace. \label{eq:u1xp}
\end{align}
It is also easy to see that
\begin{align}
  u_0(x_\ell) &= d_1(x_\ell) \enspace, \label{eq:u0xell} \\
  u_{-1}(x_r) &= d_0(x_r)+1 \enspace. \label{eq:um1xr}
\end{align}
Combining our previous observations we obtain
\begin{equation}
\label{eq:u1xrp}
  u_1(x_r') \eqByM{\eqref{eq:u1xr},\eqref{eq:um1xr}} d_0(x_r)+1 \eqByM{\eqref{eq:d0xlr},\eqref{eq:d0x}} d_1(x_\ell) \eqByM{\eqref{eq:u0xell},\eqref{eq:xell}} u_0(x_\ell')
\end{equation}
and therefore
\begin{equation}
\label{eq:u1xp-result}
  u_1(x') \eqByM{\eqref{eq:u1xp},\eqref{eq:u1xrp}} u_1(x_\ell')+u_0(x_\ell') \enspace.
\end{equation}
The reasoning in the case where the $\downstep$-step at position $i$ of $x$ starts on the line $y=1$ can be obtained by straightforward modifications of the above equations. Specifically, we have to subtract 1 on the right-hand sides of \eqref{eq:d0x} and \eqref{eq:um1xr}, and add 1 on the right-hand sides of \eqref{eq:u1xp} and \eqref{eq:u0xell}, yielding the same resulting relation \eqref{eq:u1xp-result}.
We obtain from \eqref{eq:u1xp-result} that the position of the $(u_1(x')+1)$-th $\upstep$-step of $x'$ touching the line $y=1$ equals the position $\pos(g,x)=i$ of the $\downstep$-step in $x$ that is flipped by $g$. Using the definition of $h'$, it follows that $g^{-1}=h'$. As both sets $L_{2k,k}\setminus D_{2k}^k$ and $L_{2k,k+1}$ have the same size $\binom{2k}{k}-\frac{1}{k+1}\binom{2k}{k}=\binom{2k}{k+1}$, the mappings $g$ and $h'$ are indeed bijections.
\end{proof}

\begin{lemma}
\label{lem:h-inverse}
The mappings $h$ and $g'$ are bijections between the sets $L_{2k,k+1}$ and $L_{2k,k}\setminus D_{2k}^0$ and we have $h^{-1}=g'$.
\end{lemma}

\begin{proof}
This proof is obtained by small modifications to the proof of Lemma~\ref{lem:g-inverse}.
As $x':=g'(x)$ is obtained by flipping the $d_0(x)$-th $\downstep$-step of $x$, we have to add $1$ on the right-hand side of \eqref{eq:d0xlr}.
Equations \eqref{eq:xell}, \eqref{eq:u1xr}, \eqref{eq:d0x}--\eqref{eq:um1xr} remain the same in their respective cases so that the resulting equation \eqref{eq:u1xp-result} changes to
\begin{equation*}
  u_1(x')=u_1(x_{\ell}')+u_0(x_{\ell}')+1 \enspace.
\end{equation*}
We conclude from this relation that the position of the $u_1(x')$-th $\upstep$-step of $x'$ touching the line $y=1$ equals the position $\pos(g',x)$ of the $\downstep$-step in $x$ that is flipped by $g'$. Using the definition of $h$, it follows that $h^{-1}=g'$. As both sets $L_{2k,k+1}$ and $L_{2k,k}\setminus D_{2k}^0$ have the same size, we obtain that $h$ and $g'$ are indeed bijections.
\end{proof}

The next two lemmas show that the two positions that are flipped when applying $g$ and $h$ consecutively (in one of the two orders) are as close to each other as possible.

\begin{lemma}
\label{lem:01-positions}
For any lattice path $x\in L_{2k,k}\setminus D_{2k}^k$ we have $\pos(h,g(x))<\pos(g,x)$ and in $x$ and $g(x)$ there is no $\upstep$-step touching the line $y=1$ at any position $i$ with $\pos(h,g(x))<i<\pos(g,x)$.
\end{lemma}

\begin{proof}
Let $x\in L_{2k,k}\setminus D_{2k}^k$ and $x':=g(x)$.
By Lemma~\ref{lem:g-inverse} we know that $h'$ is the inverse mapping of $g$. This means that the $(u_1(x')+1)$-th $\upstep$-step of $x'$ touching the line $y=1$ is at position $\pos(g,x)$. The mapping $h$ however already flips the $u_1(x')$-th $\upstep$-step of $x'$ touching the line $y=1$ (the previous one with that property), and the lemma follows.
\end{proof}

\begin{lemma}
\label{lem:10-positions}
For any lattice path $x\in L_{2k,k+1}$ we have $\pos(h,x)<\pos(g,h(x))$ and in $h(x)$ there is no $\downstep$-step touching the line $y=0$ at any position $i$ with $\pos(h,x)<i<\pos(g,h(x))$.
\end{lemma}

\begin{proof}
Let $x\in L_{2k,k+1}$ and $x':=h(x)$.
By Lemma~\ref{lem:h-inverse} we know that $g'$ is the inverse mapping of $h$. This means that the $d_0(x')$-th $\downstep$-step of $x'$ touching the line $y=0$ is at position $\pos(h,x)$. The mapping $g$ however flips the $(d_0(x')+1)$-th $\downstep$-step of $x'$ touching the line $y=0$ (the next one with that property), and the lemma follows.
\end{proof}

For any Dyck path $x\in D_{2k}^e$, $e<k$, let $b:=\pos(g,x)$ and $a:=\pos(h,g(x))$, respectively, be the positions of the $\downstep$-step and $\upstep$-step that are flipped when applying $f$. We refer to the subpath of $x$ strictly to the left of position $a$ as the \emph{initial part}, to the subpath strictly between positions $a$ and $b$ as the \emph{middle part}, and to the subpath strictly to the right of position $b$ as the \emph{terminal part}.
As mentioned before, $f(x)$ is obtained from $x$ by flipping the steps at positions $a$ and $b$, leaving the initial and terminal part of $x$ unchanged, and shifting down the middle part down by two units. In particular, no step of $x$ is shifted up by $f$.

The next lemma shows that our minimum-change bijection increases the number of flaws by exactly~1.

\begin{lemma}
\label{lem:flaws++}
For any $x\in D_{2k}^e$, $e<k$, we have that $f(x)\in D_{2k}^{e+1}$. 
\end{lemma}

\begin{proof}
Each Dyck path with $e$ flaws has exactly $e$ many $\downstep$-steps and exactly $e$ many $\upstep$-steps below the line $y=0$.
To prove the lemma it therefore suffices to show that for each $x\in D_{2k}^e$, the image $f(x)$ has exactly $e+1$ many $\upstep$-steps below the line $y=0$.

First note that any $\upstep$-steps of $x$ below the line $y=0$ is also below the line $y=0$ in $f(x)$: This is because the $\upstep$-step of $x$ flipped by $f$ (it touches the line $y=1$ by definition) lies above the line $y=0$ (in $x$), and because $f$ does not shift up any steps. In contrast to that, the $\downstep$-step of $x$ that is flipped by $f$ (it touches the line $y=0$ by definition) creates an additional $\upstep$-step of $f(x)$ below the line $y=0$.
It remains to argue that shifting down the middle part of $x$ does not create additional $\upstep$-steps below the line $y=0$.
To see this note that such an $\upstep$-step in the middle part of $x$ would touch the line $y=1$ (in $x$ before the shift), which is impossible because of Lemma~\ref{lem:01-positions}. This completes the proof.
\end{proof}

The following lemma is illustrated in Figure~\ref{fig:algo} below. It allows us to recognize for a given Dyck path $x\in D_{2k}^e$ the corresponding preimage $x^0\in D_{2k}^0$ such that $f^e(x^0)=x$, and will be used later in our algorithm that computes repeated applications of $f$ efficiently.

\begin{lemma}
\label{lem:recover-x0}
For any $x\in D_{2k}^e$, let $U_x$ be the set of all positions of $\upstep$-steps of $x$ that lie below the line $y=0$, and let $D_x$ be the set of all positions of $\downstep$-steps that lie below the line $y=-1$ or that are among the first $d_0(x)$ many $\downstep$-steps touching the line $y=0$. Moreover, let $x^0\in D_{2k}^0$ be such that $f^e(x^0)=x$. Then the Dyck paths $x^0$ and $x$ differ exactly in the positions $U_x\cup D_x$ (in particular, $|U_x|=|D_x|=e$).
\end{lemma}

\begin{proof}
We argue by induction over the number of flaws.
The lemma clearly holds for any $x\in D_{2k}^0$ (in this case $U_x=D_x=\emptyset$), settling the induction basis.
For the induction step we fix some $x\in D_{2k}^e$, $e<k$, for which the statement holds, and show that it also holds for $x':=f(x)$.

Let $x^0\in D_{2k}^0$ be such that $f^e(x^0)=x$.
Moreover, let $a$ and $b$, respectively, be the positions of the $\upstep$-step and $\downstep$-step of $x=(x_1,x_2,\ldots,x_{2k})$ that are flipped by $f$. By the definition of $f$ we know that $x_b$ is the $(d_0(x)+1)$-th $\downstep$-step of $x$ touching the line $y=0$, and hence $b\notin D_x$, i.e., both $x$ and $x^0$ have a $\downstep$-step at position $b$.
As argued in the proof of Lemma~\ref{lem:flaws++}, when applying $f$ to $x$ the downstep $x_b$ of $x$ is the only additional $\upstep$-step of $x'$ below the line $y=0$ (all others remain the same), implying that $U_{x'}=U_x\cup \{b\}$.
By induction we know that $U_x$ contains the positions of all $\upstep$-steps of $x$ where $x^0$ has a $\downstep$-step. Using that both $x$ and $x^0$ have a $\downstep$-step at position $b$, it follows that $U_{x'}$ contains the positions of all $\upstep$-steps of $x'$ where $x^0$ has a $\downstep$-step.

We proceed to show that $D_{x'}=D_x\cup \{a\}$ by showing both directions of this inclusion.
We begin by showing that $D_x\cup\{a\}\seq D_{x'}$.
Fix an element $c\in D_x$. If the $\downstep$-step $x_c$ at position $c$ of $x$ lies below the line $y=-1$, then it is not flipped by $f$, and also not shifted up, implying that $c\in D_{x'}$ in this case.
Otherwise, the $\downstep$-step $x_c$ is among the first $d_0(x)$ many $\downstep$-steps of $x$ touching the line $y=0$.
If $x_c$ belongs to the middle part of $x$, then it is shifted down by two units when applying $f$, so the corresponding $\downstep$-step $x'_c$ lies below the line $y=-1$, yielding $c\in D_{x'}$.
It remains to consider the case that $x_c$ belongs to the initial part of $x$.
By Lemma~\ref{lem:h-inverse} and the definition of $g'$, the $\downstep$-step $x'_a$ (which is an $\upstep$-step in $x$ that is flipped by $f$) is the $d_0(x')$-th $\downstep$-step of $x'$ touching the line $y=0$ (*), implying that the $\downstep$-step $x'_c$ is among the first $d_0(x')$ many $\downstep$-steps of $x'$, i.e., we have $c\in D_{x'}$. 
From (*) we also conclude that $a\in D_{x'}$.
We therefore obtain $D_x\cup\{a\}\subseteq D_{x'}$, as desired.
 
We now show that $D_{x'}\seq D_x\cup\{a\}$.
Let $c\in D_{x'}\setminus\{a\}$. If the corresponding $\downstep$-step $x'_c$ lies below the line $y=-1$, then the corresponding $\downstep$-step $x_c$ either also lies below the line $y=-1$, in which case we clearly have $c\in D_x$, or the $\downstep$-step $x_c$ belongs to the middle part of $x$ that is shifted down when applying $f$ and touches the line $y=0$ in $x$. By the definition of $f$ we have $c\in D_x$ also in this case.
We now consider the case that the $\downstep$-step $x'_c$ lies above the line $y=-1$ in $x'$.
Then $x'_c$ belongs to the first $d_0(x')$ many $\downstep$-steps of $x'$ touching the line $y=0$.
Using Lemma~\ref{lem:h-inverse} and the definition of $g'$, it follows that the corresponding $\downstep$-step $x_c$ belongs to the initial part of $x$ and is therefore among the first $d_0(x)$ many $\downstep$-steps of $x$ touching the line $y=0$. We therefore have $c\in D_x$ in this case as well.
Altogether, we conclude that $D_{x'}=D_x\cup \{a\}$, as claimed.
 
It remains to verify that $D_{x'}$ contains the positions of all $\downstep$-steps of $x'$ where $x^0$ has an $\upstep$-step.
By induction we know that $D_x$ contains the positions of all $\downstep$-steps of $x$ where $x^0$ has an $\upstep$-step.
As applying $f$ introduces only one additional $\downstep$-step at position $a$, we only need to check that $x^0$ has an $\upstep$-step at position $a$. Recall that the $\upstep$-step $x_a$ of $x$ touches the line $y=1$. It follows that $a\notin U_x$, and so we obtain by induction that $x^0$ indeed has an $\upstep$-step at position $a$.
This completes the proof.
\end{proof}

We now consider the sequence of flipped positions when applying the minimum-change bijection $f$ repeatedly.
Specifically, for any $x\in D_{2k}^0$, we define $x^0:=x$ and $y^i:=g(x^i)$ and $x^{i+1}:=h(y^i)$ for $i=0,1,\ldots,k-1$ (by the definition of $f$ we have $x^i=f^i(x)$ for $i=0,1,\ldots,k$) and define the sequence $p(x):=(x^0,y^0,x^1,y^1,\ldots,x^{k-1},y^{k-1},x^k)$ and the corresponding sequence of flipped positions $\pi(x):=\big(\pos(g,x^0),\pos(h,y^0),\pos(g,x^1),\pos(h,y^1),\ldots,\pos(g,x^{k-1}),\pos(h,y^{k-1})\big)$.
We call a sequence of integers $(\alpha_1,\alpha_2,\ldots,\alpha_{2k})$ \emph{alternating} if we have $\alpha_{i-1}>\alpha_i$ for all $i\in\{2,4,\ldots,2k\}$ and $\alpha_{i-1}<\alpha_i$ for all $i\in\{3,5,\ldots,2k-1\}$.
For any lattice path $x\in L_{2k,k}$, we denote by $\ol{x}$ the lattice path obtained by mirroring $x$ at the line $y=0$ (note that $x\in D_{2k}^e$ if and only if $\ol{x}\in D_{2k}^{k-e}$).

The following lemma shows that when applying $f$ repeatedly to some Dyck path with zero flaws, then the order of flipped positions alternate and each step of the path will be flipped exactly once.

\begin{lemma}
\label{lem:altperm}
For any $x\in D_{2k}^0$, the sequence $\pi(x)$ is an alternating permutation of $\{1,2,\ldots,2k\}$. In particular, we have $f^k(x)=\ol{x}$.
\end{lemma}

The alternating permutations $\pi(x)$, $x\in D_{2k}^0$, are shown at the bottom of Figure~\ref{fig:min-change-bij} for $k=3$. Note also that the Dyck paths in the bottom row of the figure (they have $k$ flaws) are mirror images of the Dyck paths in the top row (they have zero flaws).

\begin{proof}
Applying Lemma~\ref{lem:recover-x0} to $x':=f^k(x)\in D_{2k}^k$ shows that $x$ and $x'$ differ in all $2k$ positions (the sets $U_{x'}$ and $D_{x'}$ satisfy $U_{x'}\cup D_{x'}=\{1,2,\ldots,2k\}$), implying that $f^k(x)=\ol{x}$.
As applying $f^k$ to $x$ flips exactly $2k$ steps in total, we conclude that $\pi(x)$ is indeed a permutation of $\{1,2,\ldots,2k\}$. Moreover, using Lemma~\ref{lem:01-positions} and Lemma~\ref{lem:10-positions} we obtain that $\pi(x)$ is alternating.
\end{proof}

The next lemma allows us to compute the entire sequence $\pi(x)$ recursively directly from $x$, and will be used later in our algorithm that computes repeated applications of the minimum-change bijection $f$ efficiently.
To state the lemma, we define for any sequence $\alpha=(\alpha_1,\alpha_2,\ldots,\alpha_n)$ the reverse sequence $\rev(\alpha):=(\alpha_n,\alpha_{n-1},\ldots,\alpha_1)$ and we denote the length of the sequence by $|\alpha|:=n$.
Moreover, for any integer $c$ we define $\alpha+c=c+\alpha:=(\alpha_1+c,\alpha_2+c,\ldots,\alpha_n+c)$.
Given a Dyck path $x\in D_{2k}^e$, we may reverse and complement the sequence of steps of $x$, yielding the Dyck path $\ol{\rev}(x)\in D_{2k}^e$, which is obtained by mirroring $x$ at the vertical line with abscissa $k$.
Given a Dyck path $x\in D_{2k}^0$, let $b$ be the position of the first $\downstep$-step of $x=(x_1,x_2,\ldots,x_{2k})$ touching the line $y=0$, and define $u:=(x_2,\ldots,x_{b-1})$ and $v:=(x_{b+1},\ldots,x_{2k})$. We refer to this decomposition $x=\upstep\circ u\circ \downstep\circ v$ as the \emph{canonical decomposition} of $x$.

\begin{lemma}
\label{lem:altperm-rec}
Given $x\in D_{2k}^0$, let $x=\upstep\circ u\circ \downstep\circ v$ be the canonical decomposition of $x$. Then the sequence $\pi(x)$ satisfies the relation $\pi(x)=\big(|u|+2,|u|+2-\pi(\ol{\rev}(u)),1,|u|+2+\pi(v)\big)$.
\end{lemma}

Lemma~\ref{lem:altperm-rec} is illustrated in Figure~\ref{fig:altperm-rec}.
Intuitively, for any Dyck path $x\in D_{2k}^0$ the sequence $\pi(x)$ can be computed by considering all the `hills' of $x$ that are located at an even height from left to right and those located at an odd height from right to left (in each recursion step, the direction changes), and by flipping the steps in each `hill' in reverse order (first the last step, then the part in the middle recursively, then the first step).

\begin{figure}
\centering
\PSforPDF{
 \psfrag{x0}{$x\in D_{2k}^0$}
 \psfrag{pix0}{$\pi(x)=(2,1,6,4,5,3,26,12,20,16,18,17,19,14,15,13,24,22,23,21,25,8,10,9,11,7)$}
 \psfrag{n01}{\footnotesize $1$}
 \psfrag{n02}{\footnotesize $2$}
 \psfrag{n03}{\footnotesize $3$}
 \psfrag{n04}{\footnotesize $4$}
 \psfrag{n05}{\footnotesize $5$}
 \psfrag{n06}{\footnotesize $6$}
 \psfrag{n07}{\footnotesize $7$}
 \psfrag{n08}{\footnotesize $8$}
 \psfrag{n09}{\footnotesize $9$}
 \psfrag{n10}{\footnotesize $10$}
 \psfrag{na1}{\footnotesize $11$}
 \psfrag{n12}{\footnotesize $12$}
 \psfrag{n13}{\footnotesize $13$}
 \psfrag{n14}{\footnotesize $14$}
 \psfrag{n15}{\footnotesize $15$}
 \psfrag{n16}{\footnotesize $16$}
 \psfrag{n17}{\footnotesize $17$}
 \psfrag{n18}{\footnotesize $18$}
 \psfrag{n19}{\footnotesize $19$}
 \psfrag{n20}{\footnotesize $20$}
 \psfrag{n21}{\footnotesize $21$}
 \psfrag{n22}{\footnotesize $22$}
 \psfrag{n23}{\footnotesize $23$}
 \psfrag{n24}{\footnotesize $24$}
 \psfrag{n25}{\footnotesize $25$}
 \psfrag{n26}{\footnotesize $26$}
 \includegraphics{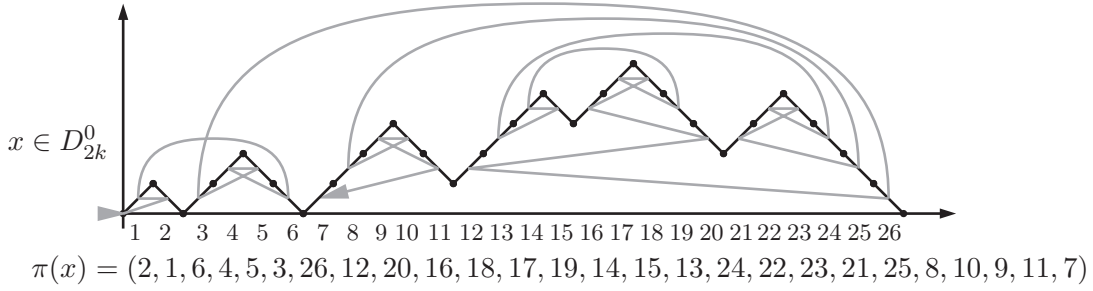}
}
\caption{Illustration of Lemma~\ref{lem:altperm-rec}. The gray line shows the order in which steps of the Dyck path $x\in D_{2k}^0$ are flipped.}
\label{fig:altperm-rec}
\end{figure}

Lemma~\ref{lem:altperm-rec} is an immediate consequence of the following more general, but slightly more technical lemma.
To state the lemma we introduce some more definitions.

Given a Dyck path $\xhat\in D_{2k}^e$, we refer to a subpath $x$ of $\xhat$ with $2\ell$ steps that starts and ends on some line $y=c$ and that has no $\downstep$-steps below the line $y=c$ as a \emph{Dyck subpath of $\xhat$}.
We use this wording to refer to a subpath of a surrounding larger Dyck path where the subpath is \emph{not} shifted to the origin.
Of course, shifting the subpath $x$ to the origin, $x$ can also be viewed as Dyck path from $D_{2\ell}^0$.
This interpretation is adopted whenever functions such as $\ol{\rev}$ or $\pi$ are applied to $x$.
Moreover, let $x$ be a Dyck subpath of $\xhat$ at positions $a+1,a+2,\ldots,a+2\ell$, and let $\sigma$ be the subsequence of $\pi(\xhat)$ consisting of exactly the elements from $\{a+1,a+2,\ldots,a+2\ell\}$.
Then we define $\pi(\xhat)|_x:=\sigma-a$ (by Lemma~\ref{lem:altperm} the sequence $\pi(\xhat)|_x$ is a permutation on the set $\{1,2,\ldots,2\ell\}$).

\begin{lemma}
\label{lem:altperm-rec+}
Let $\xhat\in D_{2k}^0$ and $c\in\mathbb{N}$, and let $x$ be a Dyck subpath of $\xhat$ at positions $a+1,a+2,\ldots,a+2\ell$ that starts and ends on the line $y=c$.
Then the elements from $\{a+1,a+2,\ldots,a+2\ell\}$ appear consecutively in $\pi(\xhat)$.
Moreover, if $c$ is even we have $\pi(\xhat)|_x=\pi(x)=\big(|u|+2,|u|+2-\pi(\ol{\rev}(u)),1,|u|+2+\pi(v)\big)$, where $x=\upstep\circ u\circ \downstep\circ v$ is the canonical decomposition of $x\in D_{2\ell}^0$.
On the other hand, if $c$ is odd we have $\pi(\xhat)|_x=|x|+1-\pi(\ol{\rev}(x))$.
\end{lemma}

\begin{proof}
For the reader's convenience, the notations used in this proof are illustrated in Figure~\ref{fig:subpath}.

Let $\xhat=(\xhat_1,\xhat_2,\ldots,\xhat_{2k})\in D_{2k}^0$ and $x=(\xhat_{a+1},\xhat_{a+2},\ldots,\xhat_{a+2\ell})$ be as in the lemma.
We prove the lemma by induction on $\ell=|x|/2$.
The statement clearly holds if $\ell=0$, settling the induction basis.
For the induction step we assume that $\ell\geq 1$ and that the statement holds for all Dyck subpaths of $\xhat$ of length strictly less than $2\ell$, and prove that is also holds for $x$.

We first consider the case that $c$ is even.
For the canonical decomposition $x=\upstep\circ u\circ \downstep\circ v$ of $x\in D_{2\ell}^0$, we consider the Dyck subpaths $u$ and $v$ of $x$ in $\xhat$ (see Figure~\ref{fig:subpath}~(a)).
 
In the following we examine the sequence of Dyck paths $\xhat,f(\xhat),\ldots,f^k(\xhat)$ and in which order the steps of the subpath $x$ are flipped when repeatedly applying $f$ (recall from Lemma~\ref{lem:altperm} that each step of $\xhat$ and $x$ is flipped exactly once).
Let $r$ be the largest integer such that the positions flipped during the first $2r$ applications of $f$ are not in the set $\{a+1,a+2,\ldots,a+2\ell\}$ (equivalently, the first $2r$ entries of $\pi(\xhat)$ are not from this set). By definition, in $f^r(\xhat)$ none of the steps at positions $a+1,a+2,\ldots,a+2\ell$ are flipped yet (w.r.t.\ $\xhat$), so the subpath $x'$ of $f^r(\xhat)$ at these positions is simply a shifted-down copy of $x$ (in particular, $x'$ is a Dyck subpath of $f^r(\xhat)$). Let $u'$ and $v'$, respectively, be the corresponding shifted-down copies of $u$ and $v$ (see Figure~\ref{fig:subpath}~(b)).

We claim that the Dyck subpath $x'$ of $f^r(\xhat)$ starts and ends on the line $y=0$ (see Figure~\ref{fig:subpath}~(b)).
To see this note that $x$ must be completely contained either in the initial, middle, or terminal part of each of the Dyck paths $\xhat,f(\xhat),\ldots,f^{r-1}(\xhat)$, and that this subpath is shifted down by two units if it is contained the middle part and not shifted otherwise.
As $c$ is even, we obtain that in each of the Dyck paths $\xhat,f(\xhat),\ldots,f^r(\xhat)$ the subpath $x$ starts and ends on a horizontal line at an even height.
Let $c'$ be this even height with respect to $f^r(\xhat)$.
Note that $c'$ cannot be negative as otherwise the first step of $x'$ (an $\upstep$-step) is an unflipped $\upstep$-step below the line $y=0$ in $f^r(\xhat)$, contradicting Lemma~\ref{lem:recover-x0}.
Moreover, as $f$ does not flip any steps above the line $y=2$ we conclude that $c'=0$, as claimed.

\begin{figure}
\centering
\PSforPDF{
 \psfrag{lab1}{(a)}
 \psfrag{lab2}{(b)}
 \psfrag{lab3}{(c)}
 \psfrag{lab4}{(d)}
 \psfrag{lab5}{(e)}
 \psfrag{xhat}{$\xhat$}
 \psfrag{x}{$x$}
 \psfrag{xp}{$x'$}
 \psfrag{u}{$u$}
 \psfrag{v}{$v$}
 \psfrag{up}{$u'$}
 \psfrag{vp}{$v'$}
 \psfrag{olu}{$\ol{u'}$}
 \psfrag{w}{$w$}
 \psfrag{f}{$f$}
 \psfrag{h}{$h$}
 \psfrag{ts}{$2s$}
 \psfrag{xpup2}{$x'_{|u|+2}$}
 \psfrag{xp1}{$x'_1$}
 \psfrag{fr}{$f^r(\xhat)$}
 \psfrag{frp1}{$f^{r+1}(\xhat)$}
 \psfrag{gfrs}{$g(f^{r+s}(\xhat))$}
 \psfrag{frsp1}{$f^{r+s+1}(\xhat)$}
 \includegraphics{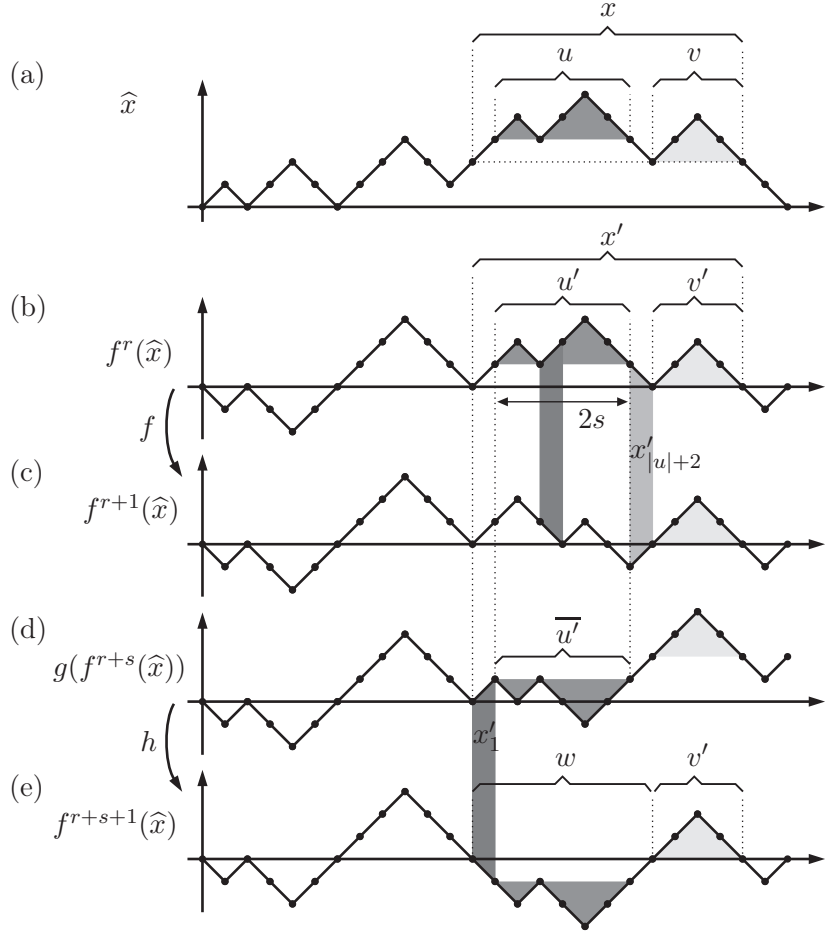}
}
\caption{Illustration of the notations used in the proof of Lemma~\ref{lem:altperm-rec+}.}
\label{fig:subpath}
\end{figure}

Next we show that applying $f$ to $f^r(\xhat)$ flips a $\downstep$-step of $x'$.
For the sake of contradiction assume that an $\upstep$-step of $x'$ and a $\downstep$-step strictly to the right of position $a+2\ell$ are flipped.
Then in $f^{r+1}(\xhat)$ the $\downstep$-step at position $a+2\ell$ (obtained by shifting down the last step of $x'$ by two units) is unflipped (w.r.t.\ $\xhat$) and lies below the line $y=-1$, contradicting Lemma~\ref{lem:recover-x0}.
Therefore, applying $f$ to $f^r(\xhat)$ indeed flips a $\downstep$-step of $x'$.

We proceed to show that when applying $f$ to $f^r(\xhat)$ the $\downstep$-step of $x'=(x'_1,x'_2,\ldots,x'_{2\ell})$ being flipped is exactly the step $x'_{|u|+2}$ (see Figure~\ref{fig:subpath}~(b) and (c)).
Clearly, the $\downstep$-step being flipped touches the line $y=0$ and the step $x'_{|u|+2}$ is the leftmost one in $x'$ with that property.
Suppose for the sake of contradiction that the $\downstep$-step $x'_{|u|+2}$ is not flipped, but instead a $\downstep$-step of $x'$ further to the right, then by the definition of $f$ the step $x'_{|u|+2}$ would be among the first $d_0(f^r(\xhat))$ many $\downstep$-steps of $f^r(\xhat)$ touching the line $y=0$.
This however contradicts Lemma~\ref{lem:recover-x0}, from which it would follow that $x'_{|u|+2}$ is flipped in $f^r(\xhat)$ (w.r.t.\ $\xhat$).

Consequently, applying $f$ to $f^r(\xhat)$ flips the $\downstep$-step $x'_{|u|+2}$ and by Lemma~\ref{lem:01-positions} the first $\upstep$-step of $f^r(\xhat)$ touching the line $y=1$ to the left of it.
As the $\upstep$-step $x'_1$ touches the line $y=1$, the $\upstep$-step being flipped is part of $u'$ if $u'$ is non-empty, and equal to $x'_1$ otherwise.

Note that $u$ is a Dyck subpath of $x$ that starts and ends on the line $y=c+1$ and has length strictly less than $2\ell$.
We therefore obtain by induction that $\pi(\xhat)|_u=|u|+1-\pi(\ol{\rev}(u))$. In particular, while no steps of $u'$ are flipped yet in $g(f^r(\xhat))$ (w.r.t.\ $\xhat$), all steps of $u'$ are flipped in $g(f^{r+s}(x))$, where $s:=|u|/2$ (see Figure~\ref{fig:subpath}~(d)).

Now it is easy to see that $x'_1$ is flipped when applying $h$ to $g(f^{r+s}(\xhat))$ (see Figure~\ref{fig:subpath}~(d) and (e)).
Indeed, in $g(f^{r+s}(x))$ the $\upstep$-step $x'_1$ is the first one touching the line $y=1$ to the left of $\ol{u'}$, so by Lemma~\ref{lem:01-positions} this step is flipped by $h$.

Let us now consider the Dyck path $h(g(f^{r+s}(x)))=f^{r+s+1}(x)$ (see Figure~\ref{fig:subpath}~(e)), and let $w$ be the subpath at positions $a+1,a+2,\ldots,a+|u|+2$ (we have $w=\downstep\circ \ol{u'}\circ \upstep$).
By Lemma~\ref{lem:10-positions}, applying $f$ to $f^{r+s+1}(x)$ flips the first $\downstep$-step touching the line $y=0$ to the right of $w$, which is clearly a $\downstep$-step of $v'$ if $v'$ is non-empty. As the length of $v'$ is strictly less than $2\ell$, we obtain by induction that $\pi(\xhat)|_v=\pi(v)$. In particular, when repeatedly applying $f$ to $f^{r+s+1}(x)$, first all steps of $v'$ are flipped, and then all other unflipped (w.r.t.\ $\xhat$) steps of $f^{r+s+1}(x)$.

Combining our previous observations we obtain $\pi(\xhat)|_x=\big(|u|+2,|u|+2-\pi(\ol{\rev}(u)),1,|u|+2+\pi(v)\big)$. Repeating the above argument with the Dyck path $\xhat:=x$ shows that this expression also equals $\pi(x)$, as claimed.

We now consider the case that $c$ is odd.
The analysis in this case is very similar to the case where $c$ is even, so we omit most of the details.
For the canonical decomposition $\ol{\rev}(x)=\upstep\circ u\circ \downstep\circ v$ of $\ol{\rev}(x)\in D_{2\ell}^0$, i.e., we have $x=\ol{\rev}(v)\circ \upstep\circ \ol{\rev}(u)\circ \downstep$, we consider the Dyck subpaths $\ol{\rev}(v)$ and $\ol{\rev}(u)$ of $x$ in $\xhat$.
We let $r$ be the maximal integer such that the positions flipped during the first $2r$ applications of $f$ to $\xhat$ are not in the set $\{a+1,a+2,\ldots,a+2\ell\}$, and let $x'$ be the subpath of $f^r(\xhat)$ at these positions ($x'$ is a shifted-down copy of $x$).
Arguing along similar lines as in the previous case, one can show that $x'=(x'_1,x'_2,\ldots,x'_{2\ell})$ starts and ends on the line $y=1$ in $f^r(x)$, and that the $\upstep$-step $x'_{|v|+1}$ is flipped when applying $f$ to $f^r(\xhat)$.
Moreover, after flipping $x'_{|v|+1}$, all the steps of the shifted-down copy of $\ol{\rev}(u)$ are flipped in the order $\pi(\ol{\rev}(u))$ (by induction), followed by the $\downstep$-step $x'_{|u|+|v|+2}=x'_{2\ell}$, followed by all the steps of the shifted-down copy of $\ol{\rev}(v)$ in the order $|v|+1-\pi(v)$ (by induction).
Combining these observations we obtain
\begin{align*}
  \pi(\xhat)|_x &= \big(|v|+1,|v|+1+\pi(\ol{\rev}(u)),|u|+|v|+2,|v|+1-\pi(v)\big) \\
                &= |x|+1-\big(|u|+2,|u|+2|-\pi(\ol{\rev}(u)),1,|u|+2+\pi(v)\big) \\
                &= |x|+1-\pi(\ol{\rev}(x)) \enspace,
\end{align*}
where in the second to last step we used $|x|=|u|+|v|+2$, and in the last step we used the first part of the lemma.
This completes the proof.
\end{proof}

\subsection{Proofs of Theorems~\texorpdfstring{\ref{thm:min-change-bij}--\ref{thm:middle-2f}}{2--5}}

With the auxiliary lemmas from the previous section in hand, we are now ready to present the proofs of our main results.

\begin{proof}[Proof of Theorem~\ref{thm:min-change-bij}]
Combining Lemmas~\ref{lem:g-inverse}, \ref{lem:h-inverse} and \ref{lem:flaws++} we obtain that for any $e\in\{0,1,\ldots,k-1\}$ the mapping $f$ is an injection between $D_{2k}^e$ and $D_{2k}^{e+1}$.
This implies $|D_{2k}^0|\leq |D_{2k}^1|\leq \cdots\leq |D_{2k}^k|$, and as $D_{2k}^0$ and $D_{2k}^k$ have the same size by symmetry, all these inequalities are in fact equalities.
In particular, $f$ is a bijection between $D_{2k}^e$ and $D_{2k}^{e+1}$.
\end{proof}

\begin{proof}[Proof of Theorem~\ref{thm:algo}]
We begin by describing an algorithm that computes for a given Dyck path $x\in D_{2k}^e$ the sequence of Dyck paths $f(x),f^2(x),\ldots,f^{k-e}(x)$, and then argue about the running time and space requirements of the algorithm.

In the initialization phase, for a given Dyck path $x\in D_{2k}^e$, our algorithm first computes the corresponding Dyck path $x^0\in D_{2k}^0$ such that $f^e(x^0)=x$ with the help of Lemma~\ref{lem:recover-x0} (see Figure~\ref{fig:algo}). We then compute the sequence $\pi(x^0)$ using the recursive rule described in Lemma~\ref{lem:altperm-rec}. With the help of $\pi(x^0)$ we can then compute each of the Dyck paths $f^i(x)$, $i=1,2,\ldots,k-e$, by flipping in the $i$-iteration the two steps in $x$ at the positions that are given by the entries of $\pi(x^0)$ at positions $2(e+i)-1$ and $2(e+i)$ (a $\downstep$-step and an $\upstep$-step, respectively).

\begin{figure}
\centering
\PSforPDF{
 \psfrag{f}{$f^e(x^0)$}
 \psfrag{x}{$x\in D_{2k}^e$}
 \psfrag{x0}{$x^0\in D_{2k}^0$}
 \psfrag{dd}{$D_x=\{1,3,4,12,13,14,16,17\}$}
 \psfrag{uu}{$U_x=\{2,5,6,15,18,19,20,26\}$}
 \psfrag{ucupd}{$U_x\cup D_x$}
 \psfrag{d0x}{$d_0(x)=4$}
 \psfrag{pix0}{$\pi(x^0)=(2,1,6,4,5,3,26,12,20,16,18,17,19,14,15,13,24,22,23,21,25,8,10,9,11,7)$}
 \includegraphics{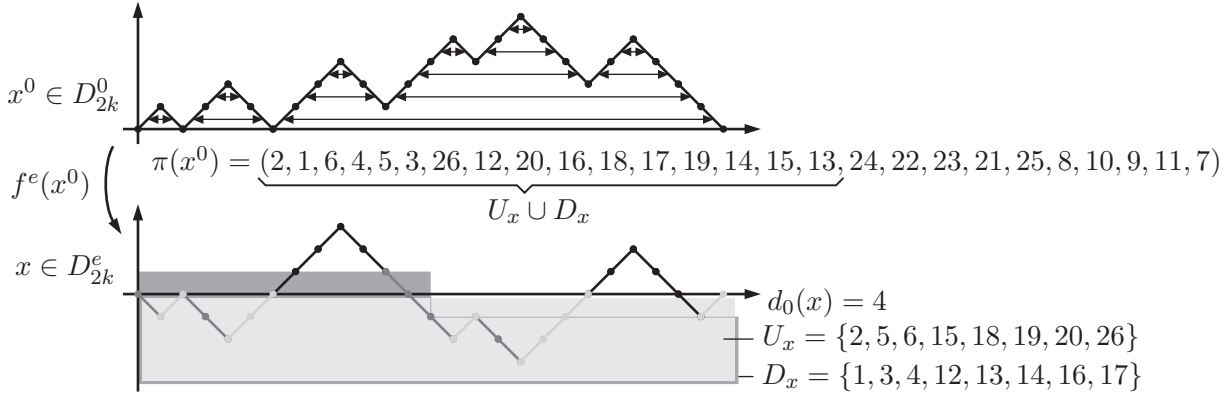}
}
\caption{Illustration of Lemma~\ref{lem:recover-x0} and of the data structures used in our algorithm from Theorem~\ref{thm:algo}. Note that $x^0$ and $x$ differ exactly at the positions in $U_x\cup D_x$. The horizontal bidirectional arrows in the upper part of the figure represent the pointer array used for fast navigation along the Dyck path $x^0$ when computing $\pi(x^0)$ recursively.}
\label{fig:algo}
\end{figure}

Computing $x^0$ can clearly be achieved in time $\cO(k)$, and with $\pi(x^0)$ at hand, each of the Dyck paths $f(x),f^2(x),\ldots,f^{k-e}(x)$ can easily be computed in time $\cO(1)$. It remains to show how to compute $\pi(x^0)$ recursively in time $\cO(k)$ (time spent in the initialization phase). For this we precompute an array of bidirectional pointers below the `hills' of $x^0$ between neighboring pairs of an $\upstep$-step and a $\downstep$-step on the same height level (see the upper part of Figure~\ref{fig:algo}).
With the help of these pointers, the canonical decomposition of $x^0$ into subpaths $u$ and $v$ required in Lemma~\ref{lem:altperm-rec} can be performed in constant time. Moreover, throughout the recursive computation of $\pi(x^0)$ we maintain the two boundary indices of the subpath currently under consideration and whether this subpath is considered in forward or backward direction (the direction alternates in each recursion step), so that the flip positions in the recursion formula can be computed directly from these boundary indices in constant time (recall the remarks after Lemma~\ref{lem:altperm-rec} and see the comments of our C++ implementation for details \cite{www}).
It follows that the overall running time for the recursion is $\cO(k)$, as desired.

The space required by all the above-mentioned data structures is clearly $\cO(k)$.
This completes the proof.
\end{proof}

\begin{proof}[Proof of Theorem~\ref{thm:odd-2f}]
For this proof we interpret the lattice paths in $L_{2k,k}$ and $L_{2k,k+1}$ as subsets of the ground set $[2k]:=\{1,2,\ldots,2k\}$, where an $\upstep$-step at position $i$ means that the element $i$ is contained in the subset, and a $\downstep$-step at position $i$ means that the element $i$ is not contained in the subset.
For any $x\in D_{2k}^0$ we consider the sequence $p(x)=(x^0,y^0,x^1,y^1,\ldots,x^{k-1},y^{k-1},x^k)$ defined before Lemma~\ref{lem:altperm}. Viewing all the $x^i$ as $k$-element subsets of $[2k]$, and the $y^i$ as $(k+1)$-element subsets of $[2k]$, by the definition of the sequence $p(x)$ we have $y^i\supset x^i$ and $y^i\supset x^{i+1}$ for all $i=0,1,\ldots,k-1$.
We now define $\ol{y^i}:=([2k]\setminus y^i)\cup\{2k+1\}$ and consider the modified sequence $c(x):=(x^0,\ol{y^0},x^1,\ol{y^1},\ldots,x^{k-1},\ol{y^{k-1}},x^k)$. The sets $x^i$ and $\ol{y^i}$ are $k$-element subsets of $[2k+1]$, and we have $\ol{y^i}\cap x^i=\emptyset$ and $\ol{y^i}\cap x^{i+1}=\emptyset$ for all $i=0,1,\ldots,k-1$.
Moreover, by Lemma~\ref{lem:altperm} we have $x^k=[2k]\setminus x^0$, so $x^0\cap x^k=\emptyset$.
We can thus interpret $c(x)$ as a cycle of length $2k+1$ in the odd graph $O_{2k+1}$ (see Figure~\ref{fig:odd-2f}). 
By Lemma~\ref{lem:g-inverse} and Lemma~\ref{lem:h-inverse}, the cycles $c(x)$, $x\in D_{2k}^0$, are all disjoint, and they must consequently form a $\cC_{2k+1}$-factor of this graph that consists of $|D_{2k}^0|=C_k$ many cycles (the total number of vertices visited by these cycles is $(2k+1)C_k$, which is equal to $\binom{2k+1}{k}$, the total number of vertices of $O_{2k+1}$).
\end{proof}

\begin{figure}
\centering
\PSforPDF{
 \psfrag{first}{$D_{2k}^0$}
 \psfrag{last}{$D_{2k}^k$}
 \psfrag{o2kp1}{\Large $O_{2k+1}$}
 \psfrag{o2v0}{$\{x\in V(O_{2k+1})\mid 2k+1\notin x\}$}
 \psfrag{o2v1}{$\{x\in V(O_{2k+1})\mid 2k+1\in x\}$}
 \psfrag{x}{$x$}
 \psfrag{xp}{$[2k]\setminus x$}
 \psfrag{cx}{$c(x)$}
 \includegraphics{graphics/odd2f.eps}
}
\caption{Illustration of the notations used in the proof of Theorem~\ref{thm:odd-2f}.}
\label{fig:odd-2f}
\end{figure}

The following simple lemma allows us to transfer cycle-factors with cycles of odd length from the odd graph to the middle levels graph, and will be used to prove Theorem~\ref{thm:middle-2f}.

\begin{lemma}
\label{lem:transfer-2f}
For any $k\geq 1$ and any odd $\ell\geq 3$, if the odd graph $O_{2k+1}$ has a $\cC_\ell$-factor, then the middle levels graph $M_{2k+1}$ has a $\cC_{2\ell}$-factor.
\end{lemma}

\begin{proof}
Let $m$ be such that $\ell=2m+1$, and for any cycle $c=:(x^1,x^2,\ldots,x^{2m+1})$ of the cycle-factor of the odd graph $O_{2k+1}$ consider the corresponding cycle
\[
 c':=(x^1,\ol{x^2},x^3,\ol{x^4},\ldots,x^{2m-1},\ol{x^{2m}},x^{2m+1},\ol{x^1},x^2,\ol{x^3},x^4,\ldots,\ol{x^{2m-1}},x^{2m},\ol{x^{2m+1}})
\]
in the middle levels graph $M_{2k+1}$, where $\ol{x^i}:=[2k+1]\setminus x^i$ for all $i=1,2,\ldots,\ell$. It is easy to see that each of these cycles has length $2\ell$, and that together they form a cycle-factor of $M_{2k+1}$ (the middle levels graph has twice as many vertices as the odd graph).
\end{proof}

\begin{proof}[Proof of Theorem~\ref{thm:middle-2f}]
Combine Theorem~\ref{thm:odd-2f} and Lemma~\ref{lem:transfer-2f}.
\end{proof}

\bibliographystyle{alpha}
\bibliography{refs}

\begin{thebibliography}{HKRR05}

\bibitem[BER76]{MR0424386}
J.~Bitner, G.~Ehrlich, and E.~Reingold.
\newblock Efficient generation of the binary reflected {G}ray code and its
  applications.
\newblock {\em Comm. ACM}, 19(9):517--521, 1976.

\bibitem[Big79]{MR556008}
N.~Biggs.
\newblock Some odd graph theory.
\newblock In {\em Second {I}nternational {C}onference on {C}ombinatorial
  {M}athematics ({N}ew {Y}ork, 1978)}, volume 319 of {\em Ann. New York Acad.
  Sci.}, pages 71--81. New York Acad. Sci., New York, 1979.

\bibitem[BR98]{MR1664740}
B.~Bultena and F.~Ruskey.
\newblock An {E}ades-{M}c{K}ay algorithm for well-formed parentheses strings.
\newblock {\em Inform. Process. Lett.}, 68(5):255--259, 1998.

\bibitem[BW84]{MR737262}
M.~Buck and D.~Wiedemann.
\newblock Gray codes with restricted density.
\newblock {\em Discrete Math.}, 48(2-3):163--171, 1984.

\bibitem[CF49]{chung-feller:49}
K.~Chung and W.~Feller.
\newblock On fluctuations in coin-tossing.
\newblock {\em Proceedings of the National Academy of Sciences of the United
  States of America}, 35(10):605--608, 1949.

\bibitem[CF02]{MR1883565}
Y.~Chen and Z.~F{\"u}redi.
\newblock Hamiltonian {K}neser graphs.
\newblock {\em Combinatorica}, 22(1):147--149, 2002.

\bibitem[CGH13]{MR3107997}
L.~Chua, A.~Gy{\'a}rf{\'a}s, and C.~Hossain.
\newblock Gallai-colorings of triples and 2-factors of {$B_3$}.
\newblock {\em Int. J. Comb.}, pages Art. ID 929565, 6 pp., 2013.

\bibitem[Cha89]{MR995888}
P.~Chase.
\newblock Combination generation and graylex ordering.
\newblock {\em Congr. Numer.}, 69:215--242, 1989.
\newblock Eighteenth Manitoba Conference on Numerical Mathematics and Computing
  (Winnipeg, MB, 1988).

\bibitem[Che00]{MR1778200}
Y.~Chen.
\newblock Kneser graphs are {H}amiltonian for {$n\geq 3k$}.
\newblock {\em J. Combin. Theory Ser. B}, 80(1):69--79, 2000.

\bibitem[Che03]{MR1999733}
Y.~Chen.
\newblock Triangle-free {H}amiltonian {K}neser graphs.
\newblock {\em J. Combin. Theory Ser. B}, 89(1):1--16, 2003.

\bibitem[Che08]{MR2382368}
Y.~Chen.
\newblock The {C}hung-{F}eller theorem revisited.
\newblock {\em Discrete Math.}, 308(7):1328--1329, 2008.

\bibitem[CL87]{MR888679}
B.~Chen and K.~Lih.
\newblock Hamiltonian uniform subset graphs.
\newblock {\em J. Combin. Theory Ser. B}, 42(3):257--263, 1987.

\bibitem[DG12]{MR2858033}
P.~Diaconis and R.~Graham.
\newblock {\em Magical mathematics}.
\newblock Princeton University Press, Princeton, NJ, 2012.
\newblock The mathematical ideas that animate great magic tricks, With a
  foreword by Martin Gardner.

\bibitem[DKS94]{MR1268348}
D.~Duffus, H.~Kierstead, and H.~Snevily.
\newblock An explicit {$1$}-factorization in the middle of the {B}oolean
  lattice.
\newblock {\em J. Combin. Theory Ser. A}, 65(2):334--342, 1994.

\bibitem[DSW88]{MR962223}
D.~Duffus, B.~Sands, and R.~Woodrow.
\newblock Lexicographic matchings cannot form {H}amiltonian cycles.
\newblock {\em Order}, 5(2):149--161, 1988.

\bibitem[EFY05]{MR2167479}
S.~Eu, T.~Fu, and Y.~Yeh.
\newblock Refined {C}hung-{F}eller theorems for lattice paths.
\newblock {\em J. Combin. Theory Ser. A}, 112(1):143--162, 2005.

\bibitem[Ehr73]{MR0366085}
G.~Ehrlich.
\newblock Loopless algorithms for generating permutations, combinations, and
  other combinatorial configurations.
\newblock {\em J. Assoc. Comput. Mach.}, 20:500--513, 1973.

\bibitem[EHR84]{MR821383}
P.~Eades, M.~Hickey, and R.~Read.
\newblock Some {H}amilton paths and a minimal change algorithm.
\newblock {\em J. Assoc. Comput. Mach.}, 31(1):19--29, 1984.

\bibitem[EM84]{MR782221}
P.~Eades and B.~McKay.
\newblock An algorithm for generating subsets of fixed size with a strong
  minimal change property.
\newblock {\em Inform. Process. Lett.}, 19(3):131--133, 1984.

\bibitem[FT95]{MR1350586}
S.~Felsner and W.~Trotter.
\newblock Colorings of diagrams of interval orders and {$\alpha$}-sequences of
  sets.
\newblock {\em Discrete Math.}, 144(1-3):23--31, 1995.
\newblock Combinatorics of ordered sets (Oberwolfach, 1991).

\bibitem[Gra]{gray:patent}
F.~Gray.
\newblock Pulse code communication.
\newblock March 17, 1953 (filed Nov. 1947). U.S. Patent 2,632,058.

\bibitem[G{\v{S}}10]{Gregor20102448}
P.~Gregor and R.~{\v{S}}krekovski.
\newblock On generalized middle-level problem.
\newblock {\em Inform. Sci.}, 180(12):2448--2457, 2010.

\bibitem[Hav83]{MR737021}
I.~Havel.
\newblock Semipaths in directed cubes.
\newblock In {\em Graphs and other combinatorial topics ({P}rague, 1982)},
  volume~59 of {\em Teubner-Texte Math.}, pages 101--108. Teubner, Leipzig,
  1983.

\bibitem[HKRR05]{horakEtAl:05}
P.~Hor\'{a}k, T.~Kaiser, M.~Rosenfeld, and Z.~Ryj\'{a}cek.
\newblock The prism over the middle-levels graph is {H}amiltonian.
\newblock {\em Order}, 22(1):73--81, 2005.

\bibitem[Hod55]{hodges:55}
J.~Hodges.
\newblock Galton's rank-order test.
\newblock {\em Biometrika}, 42(1/2):261--262, 1955.

\bibitem[HW78]{MR510592}
K.~Heinrich and W.~Wallis.
\newblock Hamiltonian cycles in certain graphs.
\newblock {\em J. Austral. Math. Soc. Ser. A}, 26(1):89--98, 1978.

\bibitem[JK04]{MR2128031}
R.~Johnson and H.~Kierstead.
\newblock Explicit 2-factorisations of the odd graph.
\newblock {\em Order}, 21(1):19--27 (2005), 2004.

\bibitem[JM95]{MR1352777}
T.~Jenkyns and D.~McCarthy.
\newblock Generating all {$k$}-subsets of {$\{1\cdots n\}$} with minimal
  changes.
\newblock {\em Ars Combin.}, 40:153--159, 1995.

\bibitem[Joh04]{MR2046083}
R.~Johnson.
\newblock Long cycles in the middle two layers of the discrete cube.
\newblock {\em J. Combin. Theory Ser. A}, 105(2):255--271, 2004.

\bibitem[Joh11]{MR2836824}
R.~Johnson.
\newblock An inductive construction for {H}amilton cycles in {K}neser graphs.
\newblock {\em Electron. J. Combin.}, 18(1):Paper 189, 12, 2011.

\bibitem[Knu11]{knuth}
D.~Knuth.
\newblock {\em The Art of Computer Programming, Volume 4A}.
\newblock Addison-Wesley, 2011.

\bibitem[KT88]{MR962224}
H.~Kierstead and W.~Trotter.
\newblock Explicit matchings in the middle levels of the {B}oolean lattice.
\newblock {\em Order}, 5(2):163--171, 1988.

\bibitem[Mac09]{MacMahon:09}
P.~MacMahon.
\newblock Memoir on the theory of the partitions of numbers. {P}art iv.
\newblock {\em Philosophical Transactions of the Royal Society of London A:
  Mathematical, Physical and Engineering Sciences}, 209(441-458):153--175,
  1909.

\bibitem[Mat76]{MR0389663}
M.~Mather.
\newblock The {R}ugby footballers of {C}roam.
\newblock {\em J. Combinatorial Theory Ser. B}, 20(1):62--63, 1976.

\bibitem[ML72]{MR0457282}
G.~Meredith and K.~Lloyd.
\newblock The {H}amiltonian graphs {$O_{4}$} to {$O_{7}$}.
\newblock In {\em Combinatorics ({P}roc. {C}onf. {C}ombinatorial {M}ath.,
  {M}ath. {I}nst., {O}xford, 1972)}, pages 229--236. Inst. Math. Appl.,
  Southend-on-Sea, 1972.

\bibitem[MN15]{MR3446435}
T.~M{\"u}tze and J.~Nummenpalo.
\newblock Efficient computation of middle levels {G}ray codes.
\newblock In {\em Algorithms---{ESA} 2015}, volume 9294 of {\em Lect. Notes
  Comp. Sci.}, pages 915--927. Springer, Heidelberg, 2015.

\bibitem[MN16]{muetze-nummenpalo:16}
T.~M{\"u}tze and J.~Nummenpalo.
\newblock A constant-time algorithm for middle levels {G}ray codes.
\newblock {\it arXiv:1606.06172}. An extended abstract has been presented at
  SODA 2017., Jul 2016.

\bibitem[MS15]{muetze-su:15}
T.~M{\"u}tze and P.~Su.
\newblock Bipartite {K}neser graphs are {H}amiltonian.
\newblock {\em Electr. Notes Discr. Math.}, 49:259 -- 267, 2015.
\newblock Proceedings of Eurocomb 2015. The full version is available at {\it
  arXiv:1503.09175} and will appear in {\it Combinatorica}.

\bibitem[M{\"u}t16]{MR3483129}
T.~M{\"u}tze.
\newblock Proof of the middle levels conjecture.
\newblock {\em Proc. London Math. Soc.}, 112(4):677--713, 2016.

\bibitem[MW12]{muetze-weber:12}
T.~M{\"u}tze and F.~Weber.
\newblock Construction of 2-factors in the middle layer of the discrete cube.
\newblock {\em J. Combin. Theory Ser. A}, 119(8):1832--1855, 2012.

\bibitem[MY09]{MR2571698}
J.~Ma and Y.~Yeh.
\newblock Generalizations of {C}hung-{F}eller theorems.
\newblock {\em Bull. Inst. Math. Acad. Sin. (N.S.)}, 4(3):299--332, 2009.

\bibitem[Nar67]{MR0224486}
T.~Narayana.
\newblock Cyclic permutation of lattice paths and the {C}hung-{F}eller theorem.
\newblock {\em Skand. Aktuarietidskr}, 1967:23--30, 1967.

\bibitem[NW75]{MR0396274}
A.~Nijenhuis and H.~Wilf.
\newblock {\em Combinatorial algorithms}.
\newblock Academic Press, New York-London, 1975.
\newblock Computer Science and Applied Mathematics.

\bibitem[RP90]{MR1041167}
F.~Ruskey and A.~Proskurowski.
\newblock Generating binary trees by transpositions.
\newblock {\em J. Algorithms}, 11(1):68--84, 1990.

\bibitem[Ruk11]{MR2776816}
J.~Rukavicka.
\newblock On generalized {D}yck paths.
\newblock {\em Electron. J. Combin.}, 18(1):Paper 40, 3 pp., 2011.

\bibitem[Rus88]{MR936104}
F.~Ruskey.
\newblock Adjacent interchange generation of combinations.
\newblock {\em J. Algorithms}, 9(2):162--180, 1988.

\bibitem[SA11]{shimada-amano}
M.~Shimada and K.~Amano.
\newblock A note on the middle levels conjecture.
\newblock {\it arXiv:0912.4564}, Sep 2011.

\bibitem[Sav93]{savage:93}
C.~Savage.
\newblock Long cycles in the middle two levels of the {B}oolean lattice.
\newblock {\em Ars Combin.}, 35-A:97--108, 1993.

\bibitem[Sav97]{MR1491049}
C.~Savage.
\newblock A survey of combinatorial {G}ray codes.
\newblock {\em SIAM Rev.}, 39(4):605--629, 1997.

\bibitem[SSS09]{MR2548541}
I.~Shields, B.~Shields, and C.~Savage.
\newblock An update on the middle levels problem.
\newblock {\em Discrete Math.}, 309(17):5271--5277, 2009.

\bibitem[Sta15]{stanley-cat:15}
R.~Stanley.
\newblock {\em Catalan numbers}.
\newblock Cambridge University Press, Cambridge, 2015.

\bibitem[SW95]{MR1329390}
C.~Savage and P.~Winkler.
\newblock Monotone {G}ray codes and the middle levels problem.
\newblock {\em J. Combin. Theory Ser. A}, 70(2):230--248, 1995.

\bibitem[TL73]{MR0349274}
D.~Tang and C.~Liu.
\newblock Distance-{$2$} cyclic chaining of constant-weight codes.
\newblock {\em IEEE Trans. Computers}, C-22:176--180, 1973.

\bibitem[Wal98]{MR1645246}
T.~Walsh.
\newblock Generation of well-formed parenthesis strings in constant worst-case
  time.
\newblock {\em J. Algorithms}, 29(1):165--173, 1998.

\bibitem[Wil89]{MR993775}
H.~Wilf.
\newblock {\em Combinatorial algorithms: an update}, volume~55 of {\em CBMS-NSF
  Regional Conference Series in Applied Mathematics}.
\newblock Society for Industrial and Applied Mathematics (SIAM), Philadelphia,
  PA, 1989.

\bibitem[Win04]{MR2034896}
P.~Winkler.
\newblock {\em Mathematical puzzles: a connoisseur's collection}.
\newblock A K Peters, Ltd., Natick, MA, 2004.

\bibitem[Woa01]{MR1840664}
W.~Woan.
\newblock Uniform partitions of lattice paths and {C}hung-{F}eller
  generalizations.
\newblock {\em Amer. Math. Monthly}, 108(6):556--559, 2001.

\bibitem[www]{www}
currently \url{http://www.math.tu-berlin.de/~muetze} and
  \url{http://www.math.tu-berlin.de/~wiechert}.

\end{thebibliography}

\end{document}